\documentclass[a4paper]{amsart}

\usepackage{amsfonts}
\usepackage{amsmath}
\usepackage{amssymb}
\usepackage{amsthm}
\usepackage[enableskew]{youngtab}
\usepackage{relsize}
\usepackage{tikz}
\usetikzlibrary{patterns}
\usetikzlibrary{decorations.pathreplacing}
\usetikzlibrary{arrows}
\usepackage{enumitem}
\usetikzlibrary{decorations.pathreplacing,angles,quotes}
\usepackage{mathtools}    
\usepackage{longtable}

\usepackage{caption}
\usepackage{subcaption}
\numberwithin{equation}{section}

\newtheorem{thm}{Theorem}[section]
\newtheorem{prop}[thm]{Proposition}
\newtheorem{lem}[thm]{Lemma}
\newtheorem{cor}[thm]{Corollary}

\newtheorem*{conj*}{Conjecture}
\newtheorem*{lem*}{Lemma}
\newtheorem*{recipe*}{Recipe}
\theoremstyle{remark}
\newtheorem{rem}[thm]{Remark}
\newtheorem{defn}[thm]{Definition}
\newtheorem{ex}[thm]{Example}
\newtheorem*{rem*}{Remark}
\newtheorem*{defn*}{Definition}
\newtheorem*{ex*}{Example}

\newcommand{\Z}{\mathbb{Z}}

\newcommand{\calS}{\mathcal{S}}
\newcommand{\calT}{\mathcal{T}}
\newcommand{\calU}{\mathcal{U}}
\newcommand{\calV}{\mathcal{V}}

\newcommand{\calX}{\mathcal{X}}
\newcommand{\calY}{\mathcal{Y}}

\newcommand{\schur}{\mathfrak{s}}
\newcommand{\elementary}{\mathfrak{e}}

\DeclareMathOperator{\sort}{sort}
\newcommand{\sorteq}{\overset{\sort}{=}}
\newcommand{\defeq}{\overset{\triangle}{=}}

\DeclareMathOperator{\LHS}{LHS}
\DeclareMathOperator{\RHS}{RHS}

\DeclareMathOperator{\sub}{sub}

\begin{document}

\title{Overlap Identities for Littlewood-Schur Functions} 
\author{Helen Riedtmann}
\thanks{This research was partially supported by the Forschungskredit of the University of Zurich, grant no. FK-15-089}


\begin{abstract} Our results revolve around a new operation on partitions, which we call overlap. We prove two overlap identities for so-called Littlewood-Schur functions. Littlewood-Schur functions are a generalization of Schur functions, whose study was introduced by Littlewood. More concretely, the Littlewood-Schur function $LS_\lambda(\calX; \calY)$ indexed by the partition $\lambda$ is a polynomial in the variables $\calX \cup \calY$ that is symmetric in both $\calX$ and $\calY$ separately. The first overlap identity represents $LS _\lambda(\calX; \calY)$ as a sum over subsets of $\calX$, while the second overlap identity essentially represents $LS_\lambda(\calX; \calY)$ as a sum over pairs of partitions whose overlap equals $\lambda$. Both identities are derived by applying Laplace expansion to a determinantal formula for Littlewood-Schur functions due to Moens and Van der Jeugt. In addition, we give two visual characterizations for the set of all pairs of partitions whose overlap is equal to a partition $\lambda$.
\end{abstract}

\maketitle

\section{Introduction}
We introduce an operation on partitions, which we call \emph{overlap}. This operation naturally leads to two identities for Schur functions. Schur functions are symmetric functions that are indexed by partitions. They are considered the most natural basis for the ring of symmetric functions owing to their connection with the irreducible representations of the symmetric group. A further reason is their orthogonality with respect to the Hall inner product, which also admits a representation theoretic interpretation. There are a variety of different definitions for Schur functions, each emphasizing one facet of this versatile symmetric function. The determinantal definition originally used by Schur is given on page \pageref{3_defn_Schur}, where we also introduce the notation $\schur_\lambda$ for the Schur function indexed by the partition $\lambda$.

We formally define the overlap of two partitions on page~\pageref{3_defn_overlap} but for now let it suffice to introduce the notation that will permit us to state the first and the second overlap identities for Schur functions: if a partition $\lambda$ is the $(m,n)$-overlap of $\mu$ and $\nu$, we write $\lambda = \mu \star_{m,n} \nu$. In addition, we denote its sign by $\varepsilon(\mu, \nu)$.

The first overlap identity states that for any set $\calX$ consisting of $m + n$ pairwise distinct variables,
\begin{align*}
\schur_{\mu \star_{m,n} \nu} (\calX) = \sum_{\substack{\calS, \calT}} \frac{\varepsilon(\mu, \nu) \schur_\mu(\calS) \schur_\nu(\calT)}{\Delta(\calS; \calT)}
\end{align*}
where the sum is over all disjoint subsets $\calS$ and $\calT$ of $\calX$ with $m$ and $n$ elements, respectively, so that their union is $\calX$. The symbol $\Delta(\calS; \calT)$ denotes the product of all pairwise differences between $s \in \calS$ and $t \in \calT$. Although the concept of overlap seems to be new, this result is not. Recalling that Schur functions can be defined as fractions of determinants, Dehaye derives this identity by expanding the determinant in the numerator with the help of Laplace expansion \cite{dehaye12}. Taking the idea of applying Laplace expansion to the determinantal definition as a starting point, we show a second overlap identity: let $\calS$ and $\calT$ be sets consisting of $m$ and $n$ variables, respectively, and set $\calX = \calS \cup \calT$. If $\Delta(\calS; \calT) \neq 0$, then
\begin{align*}
\schur_\lambda(\calX) ={} & \sum_{\substack{\mu, \nu: \\ \mu \star_{m,n} \nu = \lambda}} \frac{\varepsilon(\mu, \nu) \schur_\mu(\calS) \schur_\nu(\calT)}{\Delta(\calS; \calT)}
.
\end{align*}
In fact, the two identities for Schur functions presented here are mere corollaries to two of our main results, namely the first and the second overlap identities for Littlewood-Schur functions. Littlewood-Schur functions are a generalization of Schur functions, whose combinatorial definition appeared for the first time in the work of Littlewood \cite{littlewood}. These functions were studied under a variety of different names: they are called hook Schur functions by Berele and Regev \cite{berele_regev}, supersymmetric polynomials by Nicoletti, Metropolis and Rota \cite{metropolis}, super-Schur functions by Brenti \cite{brenti}, and Macdonald denotes them $s_\lambda(x/y)$ \cite[p.~58ff]{mac}. 
We follow Bump and Gamburd in calling them Littlewood-Schur functions and denoting them $LS_\lambda(\calX, \calY)$ \cite{bump06}. Given two sets of variables $\calX$ and $\calY$, the Littlewood-Schur function associated to a partition $\lambda$ is defined by
\begin{align*}
LS_\lambda(\calX; \calY) = \sum_{\mu,\nu} c^\lambda_{\mu \nu} \schur_\mu(\calX) \schur_{\nu'}(\calY)
\end{align*}
where $c^\lambda_{\mu \nu}$ are Littlewood-Richardson coefficients. In \cite{vanderjeugt}, Moens and Van der Jeugt show that Littlewood-Schur functions can also be described by a determinantal formula, which forms the basis for the derivation of the overlap identities.

Dropping a few technical conditions, the first and the second overlap identities for Littlewood-Schur functions have the following prerequisites: let $\calX$ and $\calY$ be sets consisting of $n$ and $m$ variables, respectively, and let $\lambda$ be a partition with $(m,n)$-index $k$. (Turn to page~\pageref{3_defn_index} for the definition of index.) The first overlap identity states that if $(\lambda_1, \dots, \lambda_{n - k}) = \mu \star_{l, n - k - l} \nu$, then 
\begin{align*}
LS_{\lambda} (-\calX; \calY) ={} & \sum_{\calS, \calT} \frac{\varepsilon(\mu, \nu) LS_{\mu + \left\langle k^l \right\rangle}(-\calS; \calY) LS_{\nu \cup (\lambda_{n + 1 - k}, \lambda_{n + 2 - k}, \dots)}(- \calT; \calY)}{\Delta(\calT; \calS)}
\end{align*}
where the sum ranges over all disjoint subsets $\calS$ and $\calT$ of $\calX$ with $l$ and $n - l$ elements, respectively, so that their union is $\calX$. An explanation of the symbols that may not be familiar to the reader can be found in Sections~\ref{3_section_sequences} and \ref{3_section_partitions}. Partitioning $\calX$ into two disjoint subsets, say $\calS$ and $\calT$, consisting of $l$ and $n - l$ variables, respectively, the second overlap identity states that
\begin{align*}
LS_{\lambda} (-\calX; \calY) ={} & \sum_{p = 0}^{\min\{l,m\}} \sum_{\substack{\calU, \calV}} \hspace{10pt} \sum_{\substack{\mu, \nu: \\ \mu \star_{l - p, n - k - l + p} \nu = (\lambda_1, \dots, \lambda_{n - k})}} \frac{\Delta(\calV; \calS) \Delta(\calT; \calU)}{\Delta(\calV; \calU) \Delta(\calT; \calS)} \\
& \times \varepsilon(\mu, \nu) LS_{\mu - \left\langle (m - k)^{l - p} \right\rangle} (-\calS; \calU) LS_{\nu \cup (\lambda_{n + 1 - k}, \lambda_{n + 2 - k}, \dots)} (-\calT; \calV)
\end{align*}
where $\calU$ and $\calV$ range over all disjoint subsets of $\calY$ with $p$ and $m - p$ elements, respectively, so that their union is equal to $\calY$.

In \cite{bump06}, Bump and Gamburd derive the simplest case of the first overlap identity for Littlewood-Schur functions (under slightly stronger assumptions). They use this identity to give a self-contained combinatorial derivation of formulas for averages of products and ratios of characteristic polynomials of random matrices from the unitary group. At the time the formulas themselves were already known but Bump and Gamburd's approach provided elegant proofs. In a subsequent paper \cite{mixed_ratios}, we will use the overlap identities presented in this paper to derive a new formula for averages of mixed ratios of characteristic polynomials of random matrices from the unitary group. Our interest in averages of this type is motivated by connections with number theory discovered by Keating and Snaith \cite{KS00zeta}. 

We also provide two visual descriptions for the notion of overlap, which are more intuitive than its formal definition. Let us illustrate the first visualization (which is formally stated in Proposition~\ref{3_prop_visual_interpretation_overlap}) on an example, as it is the reason why this operation is called the $(m,n)$-overlap of two partitions. Fix two non-negative integers $m$ and $n$, say 3 and 5, as well as a partition of length at most $m + n = 8$, say $\lambda = (4, 2, 2, 2, 2, 1, 0, 0)$. Notice that we have appended zeros, ensuring that $\lambda$ is represented by a sequence of length exactly 8. Pairs of partitions whose $(m,n)$-overlap equals $\lambda$ may be viewed as staircase walks in an $n \times m$ rectangle whose steps are labeled by the parts of $\lambda$. Let us consider the following diagram of a labeled staircase walk $\pi$ in a $5 \times 3$ rectangle:
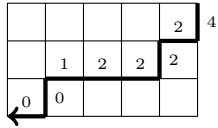
\begin{center}
\begin{tikzpicture} 
\draw[step=0.5cm, thin] (0, 0) grid (2.5, 1.5);
\draw[ultra thick, ->] (0.5, 0) -- (0, 0);
\draw[ultra thick] (0.5, 0) -- (0.5, 0.5);
\draw[ultra thick] (0.5, 0.5) -- (2, 0.5);
\draw[ultra thick] (2, 0.5) -- (2, 1);
\draw[ultra thick] (2, 1) -- (2.5, 1);
\draw[ultra thick] (2.5, 1) -- (2.5, 1.5);
\node[anchor=west] at (2.5, 1.25) {\tiny{4}};
\node[anchor=south] at (2.25, 1) {\tiny{2}};
\node[anchor=west] at (2, 0.75) {\tiny{2}};
\node[anchor=south] at (1.75, 0.5) {\tiny{2}};
\node[anchor=south] at (1.25, 0.5) {\tiny{2}};
\node[anchor=south] at (0.75, 0.5) {\tiny{1}};
\node[anchor=west] at (0.5, 0.25) {\tiny{0}};
\node[anchor=south] at (0.25, 0) {\tiny{0}};
\end{tikzpicture} 
\end{center}
We always walk down stairs that connect the top-right and the bottom-left corners of a rectangle. By means of this diagram, we can now describe the pair of partitions $\mu$, $\nu$ that correspond to the staircase walk $\pi$, thus giving an example of overlap. In fact, we will visualize the two partitions with the help of Ferrers diagrams, which are defined on page \pageref{3_page_Ferrers_diagram}. The partition $\mu$ is specified by the diagram on the right-hand side:
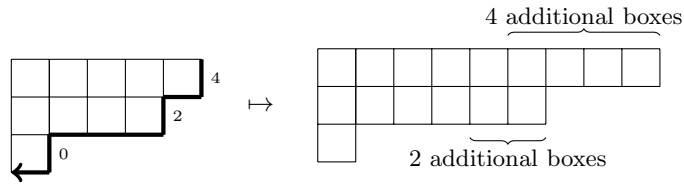
\begin{center}
\begin{tikzpicture} 
\draw[step=0.5cm, thin] (0, 0) grid (0.5, 1.5);
\draw[step=0.5cm, thin] (0.5, 0.5) grid (2, 1.5);
\draw[step=0.5cm, thin] (2, 1) grid (2.5, 1.5);
\draw[ultra thick, ->] (0.5, 0) -- (0, 0);
\draw[ultra thick] (0.5, 0) -- (0.5, 0.5);
\draw[ultra thick] (0.5, 0.5) -- (2, 0.5);
\draw[ultra thick] (2, 0.5) -- (2, 1);
\draw[ultra thick] (2, 1) -- (2.5, 1);
\draw[ultra thick] (2.5, 1) -- (2.5, 1.5);
\node[anchor=west] at (2.5, 1.25) {\tiny{4}};
\node[anchor=west] at (2, 0.75) {\tiny{2}};
\node[anchor=west] at (0.5, 0.25) {\tiny{0}};
\end{tikzpicture} 
\begin{tikzpicture} 
\node(arrow) at (-0.75, 0.74) {$\mapsto$};
\draw[step=0.5cm, thin] (0, 0) grid (0.5, 1.5);
\draw[step=0.5cm, thin] (0.5, 0.5) grid (3, 1.5);
\draw[step=0.5cm, thin] (3, 0.999) grid (4.5, 1.5);
\draw[decoration={brace, raise=5pt},decorate] (2.5, 1.5) -- node[above=6pt] {\small{4 additional boxes}} (4.5, 1.5);
\draw[decoration={brace, raise=5pt, mirror},decorate] (2, 0.5) -- node[below=6pt] {\small{2 additional boxes}} (3, 0.5);
\end{tikzpicture} 
\end{center}
In words, the labels attached to the vertical steps of $\pi$ indicate how many boxes must be added to each row of the Ferrers diagram consisting of the boxes that lie above $\pi$ in order to obtain the Ferrers diagram of the partition $\mu$. We thus get $\mu = (9,6,1)$. An analogous construction gives the Ferrers diagram of $\nu$:
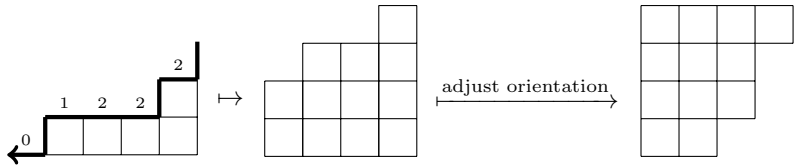
\begin{center}
\begin{tikzpicture} 
\draw[step=0.5cm, thin] (0.5, 0) grid (2, 0.5);
\draw[step=0.5cm, thin] (2, 0) grid (2.5, 1);
\draw[ultra thick, ->] (0.5, 0) -- (0, 0);
\draw[ultra thick] (0.5, 0) -- (0.5, 0.5);
\draw[ultra thick] (0.5, 0.5) -- (2, 0.5);
\draw[ultra thick] (2, 0.5) -- (2, 1);
\draw[ultra thick] (2, 1) -- (2.5, 1);
\draw[ultra thick] (2.5, 1) -- (2.5, 1.5);
\node[anchor=south] at (2.25, 1) {\tiny{2}};
\node[anchor=south] at (1.75, 0.5) {\tiny{2}};
\node[anchor=south] at (1.25, 0.5) {\tiny{2}};
\node[anchor=south] at (0.75, 0.5) {\tiny{1}};
\node[anchor=south] at (0.25, 0) {\tiny{0}};
\end{tikzpicture} 
\begin{tikzpicture} 
\node(arrow) at (-0.45, 0.74) {$\mapsto$};
\draw[step=0.5cm, thin] (0, 0) grid (0.5, 1);
\draw[step=0.5cm, thin] (0.5, 0) grid (1.5, 1.5);
\draw[step=0.5cm, thin] (1.499, 0) grid (2, 2);
\end{tikzpicture} 
\begin{tikzpicture} 
\node(arrow) at (-1.5, 0.85) {$\xmapsto{\text{adjust orientation}}$
};
\draw[step=0.5cm, thin] (0, 0) grid (1, 2);
\draw[step=0.5cm, thin] (1, 0.5) grid (1.5, 2);
\draw[step=0.5cm, thin] (1.5, 1.499) grid (2, 2);
\end{tikzpicture} 
\end{center}
The only difference is that the resulting collection of boxes needs to be rotated by 180 degrees and then transposed to adjust their orientation. We conclude that the $(3,5)$-overlap of the partitions $\mu = (9, 6, 1)$ and $\nu = (4, 3, 3, 2)$ is $\lambda = (4, 2, 2, 2, 2, 1)$; in symbols, $\mu \star_{m,n} \nu = \lambda$. In brief, the $(3,5)$-overlap of a pair of partitions encodes which boxes must be deleted in order to assemble the diagrams of both partitions into a $5 \times 3$ rectangle, \textit{i.e.}\ it keeps track of where the two partitions overlap. In addition, we define the sign of an overlap to be $(-1)$ to the power of the number of boxes below the staircase walk that corresponds to it. In our example the sign is thus given by $\varepsilon(\mu, \nu) = (-1)^5$. One application of this visualization for two overlapping partitions is a new proof of the dual Cauchy identity. 

\subsection{Structure of this paper}
In Section~\ref{3_section_background_and_notation}, we give the required background on partitions, Schur functions and Littlewood-Schur functions. In Section~\ref{3_section_lapalace_expansion_for_LS}, we introduce the formal definition of overlap and then use Laplace expansion to derive two overlap identities for Littlewood-Schur functions. Section~\ref{3_section_visualizations_of_overlap} contains two visual characterizations for the set of all pairs of partitions with the same overlap. Applying these visualizations to the second overlap identity results in more overlap identities, which are listed in Section~\ref{3_section_more_overlap_identities}.

\section{Background and notation} \label{3_section_background_and_notation}
Before presenting the required combinatorial background, let us fix some general notation: we
use the symbol $\defeq$ to denote an equality between the quantities on its left-hand side and its right-hand side which defines the quantity on its left-hand side. Furthermore,
$\LHS$/$\RHS$ always denotes the left-hand/right-hand side of the equality under consideration.

\subsection{Sequences and sets of variables} \label{3_section_sequences}
Throughout this paper a sequence will be a \emph{finite} enumeration of elements, such as \label{symbol_sequence} $\calX = \left(\calX_1, \dots, \calX_n \right)$. Its length is the number of its elements, denoted by \label{symbol_length_of_sequence} $l(\calX) = n$. A subsequence $\calY$ of $\calX$ is a sequence given by $\calY_k = \calX_{n_k}$ where $1 \leq n_1 < n_2 < \dots \leq n$ is an increasing sequence of indices; in other words, if $K$ is a subsequence of \label{symbol_n_in_square_brackets} $[n] = (1,2, \dots, n)$, then the $K$-subsequence of $\calX$ is given by \label{symbol_subsequence_indexed_by_K} $\calX_K = \left(\calX_{K_1}, \dots, \calX_{K_{l(K)}}\right)$. We denote the complement of a subsequence \label{symbol_subsequence} $\calY \subset \calX$ by \label{symbol_complement_of_subsequence} $\calX \setminus \calY \subset \calX$. The union of two sequences \label{symbol_union_of_sequences} $\calX \cup \calY$ is obtained by appending $\calY$ to $\calX$; we sometimes add subscripts to indicate the lengths of the two sequences in question. All other operations on sequences, such as \label{symbol_sum_of_sequences} addition, are understood to be element wise.

For sequences whose elements lie in a ring, we define the following two functions: \label{3_page_Delta}
\begin{align*}
\Delta(\calX) = \prod_{1 \leq i < j \leq n} (\calX_i - \calX_j) \:\text{ and }\: \Delta(\calX; \calY) = \prod_{x \in \calX, y \in \calY} (x - y)
.
\end{align*}
If two sequences $\calX$ and $\calY$ of the same length are equal up to reordering their elements, we write \label{symbol_sorteq} $\calX \sorteq \calY$. In that case $\Delta(\calX) = \varepsilon(\sigma)\Delta(\calY)$, where $\varepsilon(\sigma)$ is the sign of the sorting permutation $\sigma$ with the property that $\sigma(\calX) = \calY$. We implicitly view all sets of variables as sequences but for simplicity of notation we will not fix the order of the variables explicitly. It is important, however, to stick to one order throughout a computation or within a formula.

\subsection{Partitions} \label{3_section_partitions}
A partition is a non-increasing sequence $\lambda = (\lambda_1, \dots, \lambda_n)$ of non-negative integers, called parts. If two partitions only differ by a sequence of zeros, we regard them as equal. By an abuse of notation, we say that the length of a partition is the length of the subsequence that consists of its positive parts. The size of a partition $\lambda$ is the sum of its parts, denoted $|\lambda|$.

\label{3_page_Ferrers_diagram} The Ferrers diagram of a partition $\lambda$ is defined as the set of points $(i,j) \in \Z \times \Z$ such that $1 \leq i \leq \lambda_j$; it is often convenient to replace the points by square boxes. Turn to page \pageref{3_ferrers_diagram} for some examples of Ferrers diagrams. We use the non-standard convention that for any partition $\lambda$ and any non-negative integer $i$, $(i, 0) \in \lambda$ and $(0, i) \in \lambda$, although the Ferrers diagram of $\lambda$ only contains points with strictly positive coordinates. Similarly, we define the $0$-th part of any partition to be infinitely large. This counter-intuitive usage will allow us to avoid case analysis during later computations. Given two partitions $\kappa$ and $\lambda$, we say that $\kappa$ is a subset of $\lambda$ if their Ferrers diagrams satisfy that containment relation. Note that $\kappa \subset \lambda$ is our shorthand for both subset and subsequence. It will be clear from the context whether we view $\kappa$ and $\lambda$ as sequences or diagrams. For instance, we will study sub\emph{sets} of partitions whose Ferrers diagram are rectangular -- not subsequences of constant sequences. We denote by $\langle m^n \rangle$ the partition $(m, \dots, m)$ of length $n$, whose Ferrers diagram is a rectangle.

The conjugate partition $\lambda'$ of $\lambda$ is given by the condition that the Ferrers diagram of $\lambda'$ is the transpose of the Ferrers diagram of $\lambda$, \textit{e.g.}\ the conjugate of $(5,5,2)$ is $(3,3,2,2,2)$. We note for later reference that if the union of two partitions $\mu$ and $\nu$ is a partition, then $(\mu \cup \nu)' = \mu' + \nu'$. 

\subsection{Schur functions} \label{3_section_schur}
We briefly present symmetric polynomials and Schur functions, following the conventions of Macdonald \cite{mac}. For each non-negative integer $r$, the $r$-th elementary symmetric polynomial is defined by
\begin{align*}
\elementary_r(\calX) = \sum_{\substack{\calY \subset \calX: \\ l(\calY) = r}} \prod_{y \in \calY} y
.
\end{align*}
We observe that for any set of variables $\calX$, the $l(\calX)$-th elementary polynomial $\elementary_{l(\calX)}(\calX)$ is just the product of all elements of $\calX$. This observation motivates the following non-standard notation:
\begin{align} \label{3_eq_e(X)_defn}
\elementary(\calX) = \prod_{x \in \calX} x
.
\end{align} 
The elementary polynomials are called symmetric because they are invariant under permutations of the elements of $\calX$. In this context, Schur functions are another family of symmetric polynomials, indexed by partitions. 

\begin{defn} [Schur functions] \label{3_defn_Schur} Let $\lambda$ be a partition and $\calX$ a set of pairwise distinct variables of length $n$. If $l(\lambda) > n$, then $\schur_\lambda(\calX) = 0$; otherwise, 
\begin{align*}
\schur_\lambda(\calX) = \frac{\det \left( x^{\lambda_j + n - j} \right)_{x \in \calX, 1 \leq j \leq n}}{\Delta(\calX)}
.
\end{align*}
This definition can be extended to sets of variables that contain repetitions given that the polynomial $\Delta(\calX)$ is a divisor of the determinant in the numerator.
\end{defn}
Using that the determinant is multilinear, one quickly checks that the Schur function $\schur_\lambda(\calX)$ is homogeneous of degree $|\lambda|$. The multilinearity of determinants also entails that for any set $\calX$ consisting of exactly $n$ variables, 
\begin{align}
\schur_{\lambda + \langle m^n \rangle} (\calX) = \elementary(\calX)^m \schur_\lambda(\calX)
.
\end{align} 

\subsection{Littlewood-Schur functions}
\begin{defn} [Littlewood-Schur functions] \label{3_defn_comb_LS} Let $\calX$ and $\calY$ be two sets of variables. For any partition $\lambda$, define
$$LS_\lambda(\calX; \calY) = \sum_{\mu, \nu} c^\lambda_{\mu \nu} \schur_\mu(\calX) \schur_{\nu'}(\calY)$$
where $c^\lambda_{\mu \nu}$ are Littlewood-Richardson coefficients; their definition can be found in \cite[p.~142]{mac}.
\end{defn}

We will not work with the combinatorial definition of Littlewood-Schur functions. Instead, the identities for Littlewood-Schur functions described in
Section~\ref{3_section_two_overlap_identities_for_LS_functions} are based on a determinantal formula for Littlewood-Schur functions discovered by Moens and Van der Jeugt [MdJ03]. In order to state their result, we need
to introduce the notion of an index of a partition.

\begin{defn} [index of a partition] \label{3_defn_index} The $(m,n)$-index of a partition $\lambda$ is the largest (possibly negative) integer $k \leq \min\{m,n\}$ that satisfies $(m + 1 - k, n + 1 - k) \not\in \lambda$. Making use of the convention introduced in Section~\ref{3_section_partitions}, it is equivalent to define the $(m,n)$-index of $\lambda$ as the smallest integer $k$ so that $(m - k, n - k) \in \lambda$.
\end{defn}

It is worth noting that this definition is not equivalent to the definition used in \cite{vanderjeugt}. They work with a similar notion of index except that it is not invariant under conjugation. 

Given the Ferrers diagram of a partition $\lambda$, its $(m,n)$-index can be read off visually:
If $(m,n) \not\in \lambda$, then $k$ is the side of the largest square with bottom-right corner $(m,n)$ that fits next to the diagram of the partition $\lambda$. If $(m,n) \in \lambda$, then $-k$ is the side of the largest square with top-left corner $(m,n)$ that fits inside the diagram of $\lambda$. Let us illustrate this by a sketch: the area colored in gray is the diagram of the partition $\lambda = (7, 4, 2, 2)$.
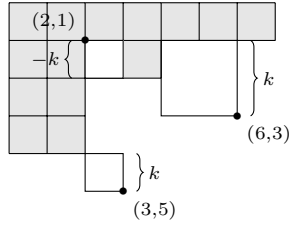
\vspace{5pt}
\begin{center}
\begin{tikzpicture}
\fill[black!10!white] (0, 0) rectangle (1, 2);
\fill[black!10!white] (1, 1) rectangle (2, 2);
\fill[black!10!white] (2, 1.5) rectangle (3.5, 2);
\draw[step=0.5cm, thin] (0, 0) grid (1, 2);
\draw[step=0.5cm, thin] (1, 0.999) grid (2, 2);
\draw[step=0.5cm, thin] (2, 1.499) grid (3.5, 2);
\draw (1, 0) rectangle (1.5, -0.5);
\draw (2, 1.5) rectangle (3, 0.5);
\fill[white] (1, 1.5) rectangle (1.5, 1);
\draw (1, 1.5) rectangle (1.5, 1);
\node[anchor=north west] at (1.5, -0.5) {$\scriptstyle{(3, 5)}$};
\node[anchor=north west] at (3, 0.5) {$\scriptstyle{(6, 3)}$};
\node[anchor=south east] at (1, 1.5) {$\scriptstyle{(2, 1)}$};
\node at (1.5, -0.5) {\tiny{\textbullet}};
\node at (3, 0.5) {\tiny{\textbullet}};
\node at (1, 1.5) {\tiny{\textbullet}};
\draw[decoration={brace, raise=5pt},decorate] (3, 1.5) -- node[right=6pt] {$\scriptstyle{k}$} (3, 0.5);
\draw[decoration={brace, raise=5pt, mirror},decorate] (1.5, -0.5) -- node[right=6pt] {$\scriptstyle{k}$} (1.5, 0);
\draw[decoration={brace, raise=5pt},decorate] (1, 1) -- node[left=6pt] {$\scriptstyle{-k}$} (1, 1.5);
\end{tikzpicture}
\end{center}
We see that the $(6,3)$-index of $\lambda$ is 2, its $(3, 5)$-index is 1, and its $(2,1)$-index is $-1$. 

\begin{thm} [determinantal formula for Littlewood-Schur functions, adapted from \cite{vanderjeugt}] \label{3_thm_det_formula_for_Littlewood-Schur}
Let $\calX$ and $\calY$ be sets of variables of length $n$ and $m$, respectively, so that the elements of $\calX \cup \calY$ are pairwise distinct. Let $\lambda$ be a partition with $(m,n)$-index $k$. If $k$ is negative, then $LS_\lambda(-\calX; \calY) = 0$; otherwise,
\begin{align*}
LS_\lambda(-\calX; \calY) ={} & \varepsilon(\lambda) \frac{\Delta(\calY; \calX)}{\Delta(\calX) \Delta(\calY)} \\
& \times \det \begin{pmatrix} \left( (x - y)^{-1} \right)_{\substack{x \in \calX \\ y \in \calY}} & \left( x^{\lambda_j + n - m - j} \right)_{\substack{x \in \calX \\ 1 \leq j \leq n - k}} \\ \left( y^{\lambda'_i + m - n - i} \right)_{\substack{1 \leq i \leq m - k \\ y \in \calY}} & 0\end{pmatrix}
\end{align*}
where $\varepsilon(\lambda) = (-1)^{\left|\lambda_{[n - k]} \right|} (-1)^{mk} (-1)^{k(k - 1)/2}$.
\end{thm}

Clearly, the sign $\varepsilon(\lambda)$ does not only depend on the partition $\lambda$, but also on the lengths of the sets of variables $\calX$ and $\calY$. However, the additional parameters $m$ and $n$ are omitted for simplicity of notation. Owing to this determinantal formula, it becomes a linear algebra exercise to check basic properties of Littlewood-Schur functions, such as that $LS_\lambda(-\calX; \calY)$ is a homogeneous polynomial of degree $|\lambda|$, which possesses the property that $LS_{\lambda'}(\calX; \calY) = LS_\lambda(\calY; \calX)$. In addition, it follows easily that Littlewood-Schur functions are symmetric in both sets of variables separately; more precisely, $LS_\lambda(-\calX; \emptyset) = \schur_\lambda(-\calX)$ and $LS_\lambda(\emptyset;\calY) = \schur_{\lambda'}(\calY)$. Theorem~\ref{3_thm_det_formula_for_Littlewood-Schur} also allows us to give new elementary proofs for some old results, such a special case of Littlewood's formula for Littlewood-Schur functions whose partition is a square.

\begin{cor} [\cite{littlewood}] \label{3_cor_littlewood} Let $\calX$ and $\calY$ be sets of variables with $n$ and $m$ elements, respectively. For any integer $l \geq 0$,
\begin{align} \label{3_cor_littlewood_eq}
LS_{\left\langle (m + l)^n \right\rangle} (-\calX; \calY) = \elementary(-\calX)^l \Delta(\calY; \calX)
.
\end{align}
\end{cor}

In this proof we will omit obvious subscripts, such as $x \in \calX$, as they clutter up the block matrices unnecessarily. Throughout this paper we take the liberty of omitting similarly intuitive subscripts during proofs whenever we feel that they are more hindrance than help.

\begin{proof} First suppose that the elements of $\calX \cup \calY$ are pairwise distinct. Given that the $(m,n)$-index of the partition $\left\langle (m + l)^n \right\rangle$ is 0,
\begin{align*}
LS_{\left\langle (m + l)^n \right\rangle}(-\calX; \calY) ={} & (-1)^{(m + l)n} \frac{\Delta(\calY; \calX)}{\Delta(\calX) \Delta(\calY)} \det \begin{pmatrix}  (x - y)^{-1} \hspace{-3.6pt} & \left( x^{n + l - j} \right)_{\substack{1 \leq j \leq n}} \\ \left( y^{m - i} \right)_{\substack{1 \leq i \leq m}} \hspace{-3.6pt} & 0\end{pmatrix}
\! .
\intertext{The two off-diagonal blocks in the matrix are squares. Hence,}
LS_{\left\langle (m + l)^n \right\rangle}(-\calX; \calY) ={} & (-1)^{(m + l)n + mn} \frac{\Delta(\calY; \calX)}{\Delta(\calX) \Delta(\calY)} \\ 
& \times \det \left( x^{n + l - j} \right)_{\substack{1 \leq j \leq n}} \det \left( y^{m - i} \right)_{\substack{1 \leq i \leq m }}
.
\intertext{Using that the determinant is multilinear, we infer that both determinants are essentially Vandermonde determinants, which cancel with $\Delta(\calX)$ and $\Delta(\calY)$, respectively. This allows us to conclude that}
LS_{\left\langle (m + l)^n \right\rangle}(-\calX; \calY) ={} & (-1)^{ln} \Delta(\calY; \calX) \elementary(\calX)^l = \Delta(\calY; \calX) \elementary(-\calX)^l
.
\end{align*}
If the elements of $\calX \cup \calY$ are not pairwise distinct, the equality in \eqref{3_cor_littlewood_eq} is a direct consequence of the fact that both sides are polynomials in $\calX \cup \calY$, which agree on infinitely many points.
\end{proof}

\section{Laplace expansion for Littlewood-Schur functions} \label{3_section_lapalace_expansion_for_LS}
In this section we present two equalities on Littlewood-Schur functions which are based on the Laplace expansion of determinants. Let us quickly recall this classical result from linear algebra. For an $n \times n$ matrix $A = \left( a_{ij} \right)$ and two subsequences $I, J \subset [n]$, we need the following notation:
\begin{align*} A_{IJ} = \left(a_{ij}\right)_{\substack{i \in I \\ j \in J}} \text{ and its complement } A_{\bar{I}\bar{J}} = \left(a_{ij}\right)_{\substack{i \not\in I \\ j \not\in J}}
.
\end{align*}

\begin{lem} [Laplace expansion] \label{3_lem_laplace_expansion_matrix} Let $A$ be an $n \times n$ matrix. For a subsequence $K \subset [n]$, the determinant of $A$ can be expanded in the two following ways:
\begin{enumerate}
\item $\displaystyle \det(A) = \sum_{\substack{J \subset [n]: \\ l(J) = l(K)}} \varepsilon(\sigma(K, J)) \det \left(A_{KJ} \right) \det \left( A_{\bar{K}\bar{J}} \right)$
\item $\displaystyle \det(A) = \sum_{\substack{I \subset [n]: \\ l(I) = l(K)}} \varepsilon(\sigma(I, K)) \det \left(A_{IK} \right) \det \left( A_{\bar{I} \bar{K}} \right)$
\end{enumerate}
where $\varepsilon(\sigma(I,J))$ is the sign of the permutation $\sigma(I,J) \in S_n$ given by the conditions that $\sigma(I) = J$ (and thus $\sigma(\bar{I}) = \bar{J}$) and that $\sigma$ respects the relative order of the indices in $I$ and $J$ as well as in $\bar{I}$ and $\bar{J}$.
\end{lem}

\subsection{Two overlap identities for Littlewood-Schur functions} \label{3_section_two_overlap_identities_for_LS_functions}
Before stating and proving the first and the second overlap identity, we formally introduce the notion of overlapping two partitions.

\begin{defn} [overlap] \label{3_defn_overlap} For any positive integer $n$, we define the partition $\rho_n$ by \label{symbol_rho_n} $\rho_n = (n - 1, \dots, 1, 0)$; for $n = 0$, we use the convention that $\rho_0$ is the empty partition. Let $\mu$, $\nu$ be partitions of length at most $m$ and $n$, respectively. The $(m,n)$-overlap of $\mu$ and $\nu$, denoted \label{symbol_overlap} $\mu \star_{m, n} \nu$, is the partition that satisfies 
\begin{align} \label{3_condition_overlap}
\mu \star_{m,n} \nu + \rho_{m + n} \sorteq (\mu + \rho_m) \cup (\nu + \rho_n)
\end{align}
if it exists; otherwise, we set $\mu \star_{m, n} \nu = \infty$. Here, $\infty$ is just a symbol with the property that $LS_\infty(\calX; \calY) = 0$ for any sets of variables $\calX$ and $\calY$, \textit{i.e.}\ it symbolizes a partition that contains the rectangle $\left\langle (m + 1)^{n + 1} \right\rangle$ for any pair of non-negative integers $m$ and $n$.

If the condition in \eqref{3_condition_overlap} is satisfied, then the sign of the overlap, which we denote by $\varepsilon_{m,n}(\mu,\nu)$, is just the sign of the corresponding sorting permutation; otherwise, we set the sign equal to 1. This notion is well defined because, unless the $(m,n)$-overlap of $\mu$ and $\nu$ is infinity, the sequence on the left-hand side in \eqref{3_condition_overlap} is strictly decreasing. If the parameters $m, n$ are clear from the context, they are sometimes omitted. 
\end{defn}

\begin{thm} [first overlap identity] \label{3_thm_laplace_expansion_LS_cut_before_index} Let $\calX$ and $\calY$ be sets of variables with $n$ and $m$ elements, respectively, so that $\calX$ consists of pairwise distinct elements. Let $\lambda$ be a partition with $(m,n)$-index $k$. If $\lambda_{[n - k]}$ is the $(l, n - k - l)$-overlap of $\mu$ and $\nu$ for some integer $0 \leq l \leq \min\{n - k, n\}$ and some partitions $\mu$ and $\nu$, then
\begin{multline}
 \label{3_thm_laplace_expansion_LS_cut_before_index_eq}
LS_{\lambda} (-\calX; \calY) = \\ \sum_{\makebox[48pt]{$\substack{\calS, \calT \subset \calX: \\ \calS \cup_{l, n - l} \calT \sorteq \calX}$}} \frac{\varepsilon_{l, n - k - l}(\mu, \nu) LS_{\mu + \left\langle k^l \right\rangle}(-\calS; \calY) LS_{\nu \cup \lambda_{(n + 1 - k, n + 2 - k, \dots)}}(- \calT; \calY)}{\Delta(\calT; \calS)} 
.
\end{multline}
\end{thm}

\begin{proof} In a first step, suppose that the variables in $\calX \cup \calY$ are also pairwise distinct, making the determinantal formula for Littlewood-Schur functions applicable. The proof boils down to writing all Littlewood-Schur functions as determinants, and then using Laplace expansion to show that the left-hand side and the right-hand side in \eqref{3_thm_laplace_expansion_LS_cut_before_index_eq} are indeed equal.

Using the definitions of both overlap and index, we determine the relevant indices of the partitions that appear on the right-hand side in \eqref{3_thm_laplace_expansion_LS_cut_before_index_eq}: The $(m,l)$-index of the partition $\mu + \left\langle k^l \right\rangle$ is 0 since $\mu_l + k \geq \lambda_{n - k} + k \geq m - k + k = m$. Furthermore, the fact that
$$\nu_{n - l - k} \geq \lambda_{n - k} \geq m - k \:\text{ and }\: \left( \nu \cup \lambda_{(n + 1 - k, \dots)} \right)_{n - l - k  + 1} = \lambda_{n + 1 - k} \leq m - k$$
implies that the $(m, n - l)$-index of $\nu \cup \lambda_{(n + 1 - k, \dots)}$ is still $k$. In particular, both sides of the equation in \eqref{3_thm_laplace_expansion_LS_cut_before_index_eq} vanish whenever $k$ is negative. Otherwise, Theorem~\ref{3_thm_det_formula_for_Littlewood-Schur} states that the right-hand side in \eqref{3_thm_laplace_expansion_LS_cut_before_index_eq} equals
\begin{align}
\begin{split} 
\RHS ={} & \sum_{\makebox[53pt]{$\substack{\calS, \calT \subset \calX: \\ \calS \cup_{l, n - l} \calT \sorteq \calX}$}} \varepsilon_{l, n - k - l}(\mu, \nu)  \varepsilon\left(\mu + \left\langle k^l \right\rangle \right) \varepsilon \left(\nu \cup \lambda_{(n + 1 - k, n + 2 - k, \dots)} \right) \\
& \times \frac{\Delta(\calY; \calS) \Delta(\calY; \calT)}{\Delta(\calT; \calS) \Delta(\calS) \Delta(\calT) \Delta(\calY)^2}
\notag \end{split} \displaybreak[2] \\
\begin{split} \label{3_in_proof_first_overlap_id_first_det}
& \times \det \begin{pmatrix} \left( (s - y)^{-1} \right) \hspace{-5pt} & \left( s^{\mu_j + k + l - m - j} \right)_{1 \leq j \leq l} \\ \left( y^{\left( \mu + \left\langle k^l \right\rangle \right)'_i + m - l - i} \right)_{1 \leq i \leq m} \hspace{-5pt} & 0 \end{pmatrix}
\end{split} \\
\begin{split} \label{3_in_proof_first_overlap_id_second_det}
& \times \det \begin{pmatrix} \left( (t - y)^{-1} \right) \hspace{-5pt} & \left( t^{\nu_j + n - l - m - j} \right)_{1 \leq j \leq n - l - k} \\ \left( y^{\left( \nu \cup \lambda_{(n + 1 - k, \dots)} \right)'_i + m - n + l - i} \right)_{1 \leq i \leq m - k} \hspace{-5pt} & 0 \end{pmatrix}
.
\end{split}
\end{align}
Focusing on the bottom-left block of the matrix in line \eqref{3_in_proof_first_overlap_id_second_det}, we observe that for $1 \leq i \leq m - k$, the exponent of $y$ in the $i$-th row is 
\begin{align*}
\text{exponent}_i(y) \defeq{} & \left( \nu \cup \lambda_{(n + 1 - k, n + 2 - k, \dots)} \right)'_i + l + m - n - i \\
={} & \nu'_i + l + \left(\lambda_{(n + 1 - k, n + 2 - k, \dots)} \right)'_i + m - n - i
.
\intertext{As $(m - k, n - l - k) \in \nu$ and $\nu'_1 = l(\nu) \leq n - l - k$, we infer that $\nu'_i = n - l - k$ for $1 \leq i \leq m - k$:}
\text{exponent}_i(y) ={} & n - k + \left(\lambda_{(n + 1 - k, n + 2 - k, \dots)} \right)'_i + m - n - i
.
\intertext{Finally, the fact that $(m - k, n - k) \in \lambda$ allows us to simplify this expression to}
\text{exponent}_i(y) ={} & \lambda'_i + m - n - i
.
\end{align*} 
Focusing on the determinant in line \eqref{3_in_proof_first_overlap_id_first_det}, we first remark that for all $1 \leq i \leq m$, $\left( \mu + \left\langle k^l \right\rangle \right)'_i = l$ because $(m,l) \in \mu + \left\langle k^l \right\rangle$ and $\mu'_1 = l(\mu) \leq l$. Hence,
\begin{align*}
\text{determinant} \defeq{} & \det \begin{pmatrix} \left( (s - y)^{-1} \right) & \left( s^{\mu_j + k + l - m - j} \right)_{1 \leq j \leq l} \\ \left( y^{\left( \mu + \left\langle k^l \right\rangle \right)'_i + m - l - i} \right)_{1 \leq i \leq m} & 0 \end{pmatrix} \\
={} & \det \begin{pmatrix} \left( (s - y)^{-1} \right) & \left( s^{\mu_j + k + l - m - j} \right)_{1 \leq j \leq l} \\ \left( y^{m - i} \right)_{1 \leq i \leq m} & 0 \end{pmatrix}
.
\intertext{As the off-diagonal blocks are squares, expanding the determinant yields}
\text{determinant} ={} & (-1)^{ml} \det \left( s^{\mu_j + k + l - m - j} \right)_{1 \leq j \leq l} \det \left( y^{m - i} \right)_{1 \leq i \leq m} \\
={} & (-1)^{ml} \det \left( s^{\mu_j + k + l - m - j} \right)_{1 \leq j \leq l}  \Delta(\calY)
.
\end{align*}
In sum, the right-hand side in \eqref{3_thm_laplace_expansion_LS_cut_before_index_eq} is equal to the following simplified determinantal expression:
\begin{align}
\begin{split}
\RHS ={} & (-1)^{ml} \varepsilon_{l, n - k - l}(\mu, \nu) \varepsilon \left(\mu + \left\langle k^l \right\rangle \right) \varepsilon \left( \nu \cup \lambda_{(n + 1 - k, \dots)} \right) \frac{\Delta(\calY; \calX)}{\Delta(\calY)} \\
& \times \sum_{\substack{\calS, \calT \subset \calX: \\ \calS \cup_{l, n - l} \calT \sorteq \calX}} \frac{1}{\Delta(\calT; \calS) \Delta(\calS) \Delta(\calT)} \det \left( s^{\mu_j + k + l - m - j} \right)_{1 \leq j \leq l} \\
& \times \det \begin{pmatrix} \left( (t - y)^{-1} \right) & \left( t^{\nu_j + n - l - m - j} \right)_{1 \leq j \leq n - l - k} \\ \left( y^{\lambda'_i + m - n - i} \right)_{1 \leq i \leq m - k} & 0 \end{pmatrix} 
.
\notag \end{split}
\intertext{Notice that $\Delta(\calT; \calS) \Delta(\calS) \Delta(\calT)$ is equal to $\Delta(\calX)$ up to a sign that measures the number of inversions with respect to the order on $\calX$. Hence, Lemma~\ref{3_lem_laplace_expansion_matrix} allows us to view this sum over $\calS$ and $\calT$ as a Laplace expansion of one determinant:}
\begin{split} \label{3_in_proof_first_overlap_id_for_cor}
\RHS ={} & (-1)^{ml + l(n + m - l)} \varepsilon_{l, n - k - l}(\mu, \nu) \varepsilon\left(\mu + \left\langle k^l \right\rangle \right) \varepsilon \left( \nu \cup \lambda_{(n + 1 - k, \dots)} \right) \\
& \times \frac{\Delta(\calY; \calX)}{\Delta(\calX) \Delta(\calY)} \\
& \times \det \begin{pmatrix} \left( (x - y)^{-1} \right) & \hspace{-14pt} \left( x^{\mu_j + k + l - m - j} \right)_{\substack{1 \leq j, \\ j \leq l}} & \hspace{-5pt} \left( x^{\nu_j + n - l - m - j} \right)_{\substack{1 \leq j, \\ j \leq n - l - k}} \\ \left( y^{\lambda'_i + m - n - i} \right)_{\substack{1 \leq i, \\ i \leq m - k}} & \hspace{-14pt} 0 & \hspace{-5pt} 0 \end{pmatrix} \hspace{-3pt}
.
\end{split}
\intertext{The condition that $\lambda_{[n - k]} = \mu \star_{l, n - k - l} \nu$ entails that permuting the $n - k$ last columns results in}
\begin{split}
\RHS ={} & (-1)^{l(n - l)} \varepsilon\left(\mu + \left\langle k^l \right\rangle \right) \varepsilon \left( \nu \cup \lambda_{(n + 1 - k, \dots)} \right) \varepsilon(\lambda) LS_\lambda(-\calX; \calY)
\notag \end{split}
\end{align}
through another application of Theorem~\ref{3_thm_det_formula_for_Littlewood-Schur}. Under the additional assumption that the variables in $\calX \cup \calY$ are pairwise distinct, the equality in \eqref{3_thm_laplace_expansion_LS_cut_before_index_eq} is thus an immediate consequence of the fact that the signs cancel each other out. If we permit that $\calX \cup \calY$ contains repetitions, it suffices to remark that for a fixed set of variables $\calX$, both sides of the equation in \eqref{3_thm_laplace_expansion_LS_cut_before_index_eq} are polynomials in $\calY$.
\end{proof}

\begin{rem} \label{3_rem_vanderjeugt_counterexample} In case the sorting of the overlap is the identity, the equation in \eqref{3_thm_laplace_expansion_LS_cut_before_index_eq} reads
\begin{align} \label{3_vanderjeugt_counterexample}
LS_{\lambda} (-\calX; \calY) ={} & \sum_{\substack{\calS, \calT \subset \calX: \\ \calS \cup_{l, n - l} \calT \sorteq \calX}} \frac{LS_{\lambda_{[l]} + \left\langle (n - l)^l \right\rangle}(-\calS; \calY) LS_{\lambda_{(l + 1, l + 2, \dots)}}(- \calT; \calY)}{\Delta(\calT; \calS)}
\end{align}
for any integer $0 \leq l \leq \min\{n - k, n \}$.

This specialization of Theorem~\ref{3_thm_laplace_expansion_LS_cut_before_index} is a slight generalization of Proposition 8 in \cite{bump06}. In fact, they prove equality \eqref{3_vanderjeugt_counterexample} under the assumption that $\lambda_l \geq \lambda_{l + 1} + m$ (and $l \leq n$). Their assumption in stronger than ours: it entails that $l \leq n - k$ where $k$ stands for the $(m,n)$-index of $\lambda$. Indeed, if $k \geq 1$, then $\lambda_i \leq m - k < m$ for all $i > n - k$ but $\lambda_l \geq m$. Bump and Gamburd's proof is an induction over $m$ based on Pieri's formula.

Independently of \cite{bump06}, Lemma 5.4 in \cite{vanderjeugt} also states equality \eqref{3_vanderjeugt_counterexample} but without any assumptions on $\lambda$. Moreover, their proof is also an induction based on Pieri's formula. However, it is possible to construct counter-examples to their claim when $l > n - k$: let us fix $n = 2$ and $m = 3$, then $\lambda = (1,1,1)$ has $(m,n)$-index $k = 2$. Setting $l = 1$, one computes that
\begin{multline*}
LS_{(1,1,1)}(-x_1, -x_2; y_1, y_2, y_3) = \\ \sum_{\makebox[36.5pt]{$\substack{\calS, \calT \subset (x_1, x_2): \\ \calS \cup_{1,1} \calT \sorteq (x_1, x_2)}$}} \frac{LS_{(1) + (1)}(-\calS; y_1, y_2, y_3) LS_{(1,1)} (-\calT; y_1, y_2, y_3)}{\Delta(\calT; \calS)} + y_1y_2y_3
\end{multline*}
despite the fact that equation \eqref{3_vanderjeugt_counterexample} does not predict the additional term $y_1y_2y_3$. The main theorem in \cite{vanderjeugt} still holds because they provide several independent proofs.
\end{rem}

\begin{cor} \label{3_cor_laplace_expansion_LS_mu_has_maximal_index} Let $\mu, \nu$ be partitions and $\calX$, $\calY$ sets of variables with $n$ and $m$ elements, respectively, so that the elements of $\calX$ are pairwise distinct. Fix an integer $l(\mu) \leq l \leq n$ and let $k$ denote the $(m, n - l)$-index of $\nu$. If $l \leq n - k$ and the $(m,l)$-index of $\mu + \left\langle k^l \right\rangle$ is $0$, then
\begin{multline}
 \label{3_cor_laplace_expansion_LS_mu_has_maximal_index_eq}
LS_{\left( \mu \star_{l, n - l - k} \nu_{[n - l - k]} \right) \cup \nu_{(n + 1 - l - k, \dots)}} (-\calX; \calY) = \\
\sum_{\substack{\calS, \calT \subset \calX: \\ \calS \cup_{l, n - l} \calT \sorteq \calX}} \frac{\varepsilon \left(\mu, \nu_{[n - l - k]} \right) LS_{\mu + \left\langle k^l \right\rangle}(-\calS; \calY) LS_\nu(- \calT; \calY)}{\Delta(\calT; \calS)}
.
\end{multline} 
\end{cor}

\begin{proof} First suppose that there exists a partition $\lambda$ with the property that $$\lambda = \left( \mu \star_{l, n - l - k} \nu_{[n - l - k]} \right) \cup \nu_{(n + 1 - l - k, \dots)}.$$ It is easy to check that the $(m, n)$-index of $\lambda$ is $k$. Hence, the equality in \eqref{3_cor_laplace_expansion_LS_mu_has_maximal_index_eq} is a direct consequence of Theorem~\ref{3_thm_laplace_expansion_LS_cut_before_index}. 

Second suppose that $\mu \star_{l, n - l - k} \nu_{[n - l - k]} = \infty$. On the one hand, the left-hand side in \eqref{3_cor_laplace_expansion_LS_mu_has_maximal_index_eq} vanishes by definition. On the other hand, condition \eqref{3_condition_overlap} implies that there exist $1 \leq p \leq l$ and $1 \leq q \leq n - l - k$ such that $\mu_p + l - p = \alpha = \nu_q + n - l - k - q$. If $k$ is negative, the right-hand side in \eqref{3_cor_laplace_expansion_LS_mu_has_maximal_index_eq} also vanishes. If $k \geq 0$, the right-hand side in \eqref{3_cor_laplace_expansion_LS_mu_has_maximal_index_eq} is equal to the following determinantal expression, owing to the arguments used to justify that the right-hand side in \eqref{3_thm_laplace_expansion_LS_cut_before_index_eq} is equal to \eqref{3_in_proof_first_overlap_id_for_cor}:
\begin{align*}
\RHS ={} & \pm \frac{\Delta(\calY; \calX)}{\Delta(\calX) \Delta(\calY)} \\
& \times \det \begin{pmatrix} \left( (x - y)^{-1} \right) & \left( x^{\mu_j + k + l - m - j} \right)_{\substack{1 \leq j \leq l}} & \left( x^{\nu_j + n - l - m - j} \right)_{\substack{1 \leq j \leq n - l - k}} \\ \ast & 0 & 0\end{pmatrix}
\end{align*}
where $\ast$ stands for some block that is not relevant here. 
Given that 
\begin{align*}
\mu_p + k + l - m - p = \alpha + k - m = \nu_q + n - l - m - q 
,\end{align*}
we conclude that the matrix contains two identical columns, which means that the right-hand side in \eqref{3_cor_laplace_expansion_LS_mu_has_maximal_index_eq} also vanishes.
\end{proof}

\begin{thm} [second overlap identity] \label{3_thm_lapalace_expansion_new_LS} Let $0 \leq l \leq \min\{n - k, n\}$. Let $\calS$, $\calT$ and $\calY$ be sets containing $l$, $n - l$ and $m$ variables, respectively, so that $\Delta(\calY) \neq 0$ and $\Delta(\calS; \calT) \neq 0$. Suppose that $k$ is the $(m,n)$-index of a partition $\lambda$, then
\begin{align}
\begin{split} \label{3_thm_lapalace_expansion_new_LS_eq}
LS_{\lambda} (-(\calS \cup \calT); \calY) ={} & \sum_{p = 0}^{\min\{l,m\}} \sum_{\substack{\calU, \calV \subset \calY: \\ \calU \cup_{p, m - p} \calV \sorteq \calY}} \hspace{10pt} \sum_{\substack{\mu, \nu: \\ \mu \star_{l - p, n - k - l + p} \nu = \lambda_{[n - k]}}} \frac{\Delta(\calV; \calS) \Delta(\calT; \calU)}{\Delta(\calV; \calU) \Delta(\calT; \calS)} \\
& \times \varepsilon(\mu, \nu) LS_{\mu - \left\langle (m - k)^{l - p} \right\rangle} (-\calS; \calU) LS_{\nu \cup \lambda_{(n + 1 - k, \dots)}} (-\calT; \calV)
.
\end{split}
\end{align}
\end{thm}

It is worth noting that for any partition $\nu$ that appears in the sum the union $\nu \cup \lambda_{(n + 1 - k, \dots)}$ is again a partition. Indeed, the fact that the $(l - p, n - k - l + p)$-overlap of $\mu$ and $\nu$ is $\lambda_{[n - k]}$ implies that there exists an index $1 \leq i \leq n - k$ so that $\nu_{n - k - l + p} = \lambda_i + n - k - i \geq \lambda_{n - k}.$

\begin{proof} We remark that if $\mu \star_{l - p, n - k - l + p} \nu = \lambda_{[n - k]}$, then the $(m - p, n - l)$-index of the partition $\nu \cup \lambda_{(n + 1 - k, \dots)}$ is $k - p$. Indeed, recalling that $k$ is the $(m,n)$-index of $\lambda$ allows us to infer that
\begin{align*} \left( \nu \cup \lambda_{(n + 1 - k, \dots)} \right)_{n - l - (k - p)} ={} & \nu_{n - k - l + p} \geq \lambda_{n - k} \geq m - k \intertext{ and } \left( \nu \cup \lambda_{(n + 1 - k, \dots)} \right)_{n - l - (k - p) + 1} ={} & \lambda_{n - k + 1} \leq m - k
.
\end{align*}
Therefore, both sides of the equation in \eqref{3_thm_lapalace_expansion_new_LS_eq} vanish whenever $k$ is negative. To prove equality for non-negative $k$, we first suppose that the variables in $\calS \cup \calT \cup \calY$ are pairwise distinct. According to Theorem~\ref{3_thm_det_formula_for_Littlewood-Schur}, the left-hand side in \eqref{3_thm_lapalace_expansion_new_LS_eq} can be written as
\begin{align*}
\LHS ={} & \varepsilon(\lambda) \frac{\Delta(\calY; \calS \cup \calT)}{\Delta(\calS \cup \calT) \Delta(\calY)} \det \begin{pmatrix} \left( (s - y)^{-1} \right) & \left( s^{\lambda_j + n - m - j} \right)_{\substack{1 \leq j \leq n - k}} \\ \left( (t - y)^{-1} \right) & \left( t^{\lambda_j + n - m - j} \right)_{\substack{1 \leq j \leq n - k}} \\ \left( y^{\lambda'_i + m - n - i} \right)_{\substack{1 \leq i \leq m - k}} & 0\end{pmatrix}
.
\intertext{Let us expand the determinant along the first $l$ rows by applying Lemma~\ref{3_lem_laplace_expansion_matrix}. This results in a signed sum over all subsets of the $m$ first columns and the $n - k$ last columns, respectively, that contain exactly $l$ columns in total.  The sum over the $m$ first columns corresponds to dividing $\calY$ into two subsequences, while the sum over the $n - k$ last columns (essentially) corresponds to dividing $\lambda$ into two subpartitions:}
\LHS ={} & \varepsilon(\lambda) \frac{\Delta(\calY; \calS) \Delta(\calY; \calT)}{\Delta(\calS \cup \calT) \Delta(\calY)} \\
& \times \sum_{p = 0}^{\min\{l,m\}} \sum_{\substack{\calU, \calV \subset \calY: \\ \calU \cup_{p, m - p} \calV \sorteq \calY}} \sum_{\substack{\mu, \nu: \\ \mu \star_{l - p, n - k - l + p} \nu = \lambda_{[n - k]}}} \frac{\varepsilon(\mu, \nu) \Delta(\calY) (-1)^{(m - p)(l - p)}}{\Delta(\calU) \Delta(\calV) \Delta(\calU; \calV)} \\
& \times \det \begin{pmatrix} \left( (s - u)^{-1} \right) & \left( s^{(\mu_j - (m - k)) + l - p - j} \right)_{\substack{1 \leq j \leq l - p}} \end{pmatrix} \\
& \times \det \begin{pmatrix} \left( (t - v)^{-1} \right) & \left( t^{\nu_j + (n - l) - (m - p) - j} \right)_{\substack{1 \leq j \leq n - k - l + p}} \\ \left( v^{\lambda'_i + m - n - i} \right)_{\substack{1 \leq i \leq m - k}} & 0 \end{pmatrix}
.
\intertext{Notice that $\mu - \left\langle (m - k)^{l - p} \right\rangle$ is a partition since $\mu_{l - p} - (m - k) \geq \lambda_{n - k} - (m - k) \geq 0$. Moreover, the same inequality shows that the $(p,l)$-index of this partition is $p$. In addition, the fact that $(m - k, n - k) \in \lambda$ and $(m - k, n - l - k + p) \in \nu$ entails that for $1 \leq i \leq m - k$, $\lambda'_i = \left( \nu \cup \lambda_{(n - k + 1, \dots)} \right)'_i + l - p$. Hence, Theorem~\ref{3_thm_det_formula_for_Littlewood-Schur} states that}
\LHS ={} & \varepsilon(\lambda) \frac{\Delta(\calY; \calS) \Delta(\calY; \calT)}{\Delta(\calS \cup \calT)} \\
& \times \sum_{p = 0}^{\min\{l,m\}} \sum_{\substack{\calU, \calV \subset \calY: \\ \calU \cup_{p, m - p} \calV \sorteq \calY}} \sum_{\substack{\mu, \nu: \\ \mu \star_{l - p, n - k - l + p} \nu = \lambda_{[n - k]}}} \frac{\varepsilon(\mu, \nu) (-1)^{(m - p)(l - p)}}{\Delta(\calU) \Delta(\calV) \Delta(\calU; \calV)} \displaybreak[2] \\
& \times \varepsilon\left(\mu - \left\langle (m - k)^{l - p} \right\rangle \right) \frac{\Delta(\calS) \Delta(\calU)}{\Delta(\calU; \calS)} LS_{\mu - \left\langle (m - k)^{l - p} \right\rangle}(-\calS; \calU) \\
& \times \varepsilon\left(\nu \cup \lambda_{(n - k + 1, \dots)} \right) \frac{\Delta(\calT) \Delta(\calV)}{\Delta(\calV; \calT)} LS_{\nu \cup \lambda_{(n - k + 1, \dots)}}(-\calT; \calV)
.
\end{align*}
Combining the different factors in front of the Littlewood-Schur functions gives the desired equality. If we weaken the assumption that $\calS \cup \calT \cup \calY$ consist of pairwise distinct elements to the condition that $\Delta(\calY) \neq 0$ and $\Delta(\calS; \calT) \neq 0$, the equality in \eqref{3_thm_lapalace_expansion_new_LS_eq} still holds as $\Delta(\calT; \calS) \times \LHS$ and $\Delta(\calT; \calS) \times \RHS$ are polynomials in $\calS \cup \calT$ for any fixed set of variables $\calY$.
\end{proof}

\subsection{Laplace expansion for Schur functions}
Given that $\schur_\lambda(\calX) = LS_\lambda(\calX; \emptyset)$, any Schur function may be viewed as a specialization of a Littlewood-Schur function. The first and second overlap identities look much neater when specialized to Schur functions. The primary reason why these statements simplify so drastically is that the $(0,n)$-index of any partition with length less than $n$ is equal to $0$. Specializing Corollary~\ref{3_cor_laplace_expansion_LS_mu_has_maximal_index} to Schur functions gives the following identity. 

\begin{cor} [first overlap identity for Schur functions, \cite{dehaye12}] \label{3_cor_laplace_expansion_POD} Let the set $\calX$ consist of $m + n$ pairwise distinct variables. For any pair of partitions $\mu$ and $\nu$ of lengths at most $m$ and $n$, respectively, it holds that
\begin{align*}
\schur_{\mu \star_{m,n} \nu} (\calX) = \sum_{\substack{\calS, \calT \subset \calX: \\ \calS \cup_{m,n} \calT \sorteq \calX}} \frac{\varepsilon(\mu, \nu) \schur_\mu(\calS) \schur_\nu(\calT)}{\Delta(\calS; \calT)}
.
\end{align*}
\end{cor}

Corollary~\ref{3_cor_laplace_expansion_POD} is nothing more than a reformulation of Lemma 5 in \cite{dehaye12}. The only notable difference is that Dehaye does not introduce the notion of overlapping two partitions. In fact, Dehaye's lemma was the starting point for this paper. Interestingly, the case $\mu_{m} \geq \nu_1 + n$ (i.e. when the sorting algorithm is the identity) appears independently in both \cite{bump06} and \cite{vanderjeugt} with essentially identical proofs.

The following corollary to Theorem~\ref{3_thm_lapalace_expansion_new_LS} is obtained by setting $\calY = \emptyset$.

\begin{cor} [second overlap identity for Schur functions] \label{3_cor_laplace_expansion_schur_new} Let $\lambda$ be a partition and let $\calS$ and $\calT$ be sets consisting of $m$ and $n$ variables, respectively. If $\Delta(\calS; \calT) \neq 0$, then
\begin{align*}
\schur_\lambda(\calS \cup \calT) ={} & \sum_{\substack{\mu, \nu: \\ \mu \star_{m,n} \nu = \lambda}} \frac{\varepsilon(\mu, \nu) \schur_\mu(\calS) \schur_\nu(\calT)}{\Delta(\calS; \calT)}
.
\end{align*}
\end{cor}

\section{Visualizing the overlap of two partitions} \label{3_section_visualizations_of_overlap}
In this section we present two visual interpretations for overlapping partitions. Both visualizations characterize the set of all pairs of partitions whose overlaps are equal by identifying their Ferrers diagrams with some part of the diagram of a so-called staircase walk.

\begin{defn} [staircase walks]  A staircase walk is a lattice walk that only uses west and south steps. Let \label{symbol_staircase_walks_in_rectangle} $\mathfrak{P}(n,m)$ be the set of all staircase walks going from the top-right to the bottom-left corner of an $n \times m$ rectangle. For a staircase walk $\pi \in \mathfrak{P}(n,m)$, \label{symbol_partition_associated_to_pi_mu} $\mu(\pi) \subset \langle n^m \rangle$ denotes the partition whose Ferrers diagram lies above $\pi$, while \label{symbol_partition_associated_to_pi_nu} $\nu(\pi) \subset \langle n^m \rangle$ denotes the partition whose Ferrers diagram (rotated by 180 degrees) lies below $\pi$. In addition, $V(\pi)$ and \label{symbol_sequence_associated_to_pi} $H(\pi)$ denote the sequences of all times of vertical and horizontal steps of $\pi$, respectively.
\end{defn}

\begin{ex*} The staircase walk $\pi \in \mathfrak{P}(6,3)$ cuts the $6 \times 3$ rectangle into the following two partitions: $\mu(\pi) = (5,5,2)$ and $\nu(\pi) = (4,1,1)$.
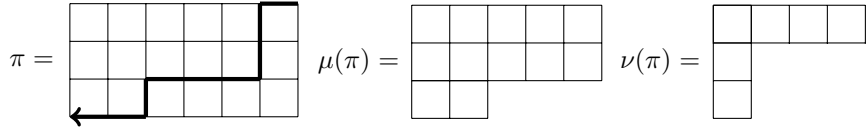
\begin{center}
\begin{tikzpicture} 
\node(pi) at (-0.5, 0.75) {$\pi =$};
\draw[step=0.5cm, thin] (0, 0) grid (3, 1.5);
\draw[ultra thick, ->] (1, 0) -- (0, 0);
\draw[ultra thick] (1, 0) -- (1, 0.5);
\draw[ultra thick] (1, 0.5) -- (2.5, 0.5);
\draw[ultra thick] (2.5, 0.5) -- (2.5, 1.5);
\draw[ultra thick] (2.5, 1.5) -- (3, 1.5);
\end{tikzpicture}
\begin{tikzpicture} \label{3_ferrers_diagram}
\node(mu) at (-0.7, 0.75) {$\mu(\pi) =$};
\draw[step=0.5cm, thin] (0, 0) grid (1, 0.5);
\draw[step=0.5cm, thin] (0, 0.5) grid (2.5, 1.5);
\end{tikzpicture}
\begin{tikzpicture}
\node(nu) at (-0.7, 0.75) {$\nu(\pi) =$};
\draw[step=0.5cm, thin] (0, 0) grid (0.5, 1.5);
\draw[step=0.5cm, thin] (0, 0.99) grid (2, 1.5);
\end{tikzpicture}
\end{center}
We further see that $V(\pi) = (2,3,7)$ and $H(\pi) = (1,4,5,6,8,9)$.
\end{ex*}

\subsection{Labeled staircase walks}
Our first visualization for the overlap of two partitions is based on the following lemma, which provides a non-visual description for the two partitions $\mu(\pi)$ and $\nu(\pi)$ associated to any staircase walk $\pi$.  

\begin{lem} \label{3_lem_macdonald_page_3} Let $\pi \in \mathfrak{P}(n,m)$, then $\mu(\pi)$ and $\nu(\pi)$ are the unique partitions that satisfy the following two equations:
$$\mu(\pi) + \rho_m = (\rho_{m + n})_{V(\pi)} \:\text{ and }\: \nu(\pi)' + \rho_n = (\rho_{m + n})_{H(\pi)}.$$
\end{lem}

\begin{proof} This proof reproduces arguments from \cite[p.~3]{mac}. Consider the following diagram of a staircase walk $\pi \in \mathfrak{P}(n,m)$ and the corresponding partition $\mu(\pi)$ (colored in gray):
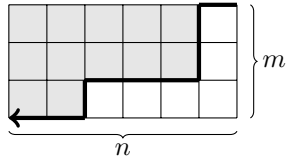
\begin{center}
\begin{tikzpicture} 
\fill[black!10!white] (0, 0) rectangle (1, 1.5);
\fill[black!10!white] (1, 0.5) rectangle (2.5, 1.5);
\draw[step=0.5cm, thin] (0, 0) grid (3, 1.5);
\draw[ultra thick, ->] (1, 0) -- (0, 0);
\draw[ultra thick] (1, 0) -- (1, 0.5);
\draw[ultra thick] (1, 0.5) -- (2.5, 0.5);
\draw[ultra thick] (2.5, 0.5) -- (2.5, 1.5);
\draw[ultra thick] (2.5, 1.5) -- (3, 1.5);
\draw[decoration={brace, raise=5pt},decorate] (3, 1.5) -- node[right=6pt] {$m$} (3, 0);
\draw[decoration={brace, raise=5pt, mirror},decorate] (0, 0) -- node[below=6pt] {$n$} (3, 0);
\node at (0, 1.7) {};
\end{tikzpicture}
\end{center}
We see that for $1 \leq i \leq m$, $V(\pi)_i = i + n - (\mu(\pi))_i$. Let us illustrate this observation for $i = 2$:
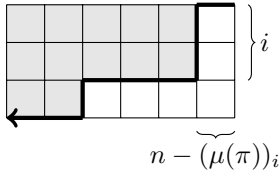
\begin{center}
\begin{tikzpicture} 
\fill[black!10!white] (0, 0) rectangle (1, 1.5);
\fill[black!10!white] (1, 0.5) rectangle (2.5, 1.5);
\draw[step=0.5cm, thin] (0, 0) grid (3, 1.5);
\draw[ultra thick, ->] (1, 0) -- (0, 0);
\draw[ultra thick] (1, 0) -- (1, 0.5);
\draw[ultra thick] (1, 0.5) -- (2.5, 0.5);
\draw[ultra thick] (2.5, 0.5) -- (2.5, 1.5);
\draw[ultra thick] (2.5, 1.5) -- (3, 1.5);
\draw[decoration={brace, raise=5pt},decorate] (3, 1.5) -- node[right=6pt] {$i$} (3, 0.5);
\draw[decoration={brace, raise=5pt, mirror},decorate] (2.5, 0) -- node[below=6pt] {$n - (\mu(\pi))_i$} (3, 0);
\end{tikzpicture}
\end{center}
In consequence,
$$m + n - V(\pi)_i = m + n - (i + n - (\mu(\pi))_i) = m - i + (\mu(\pi))_i,$$
which shows the first equality. By symmetry the analogue holds for $\nu(\pi)'$.
\end{proof}

\begin{prop} \label{3_prop_visual_interpretation_overlap} For a fixed partition $\lambda$ of length at most $m + n$, there is a 1-to-1 correspondence between $\mathfrak{P}(n,m)$ and $\{(\mu, \nu): \mu \star_{m,n} \nu = \lambda\}$
given by
\begin{align} \label{3_eq_visual_interpretation_overlap_map}
\pi \mapsto \left(\mu(\pi) + \lambda_{V(\pi)}, \nu(\pi)' + \lambda_{H(\pi)}\right).
\end{align}
Moreover, $\varepsilon_{m,n} \left(\mu(\pi) + \lambda_{V(\pi)}, \nu(\pi)' + \lambda_{H(\pi)}\right) = (-1)^{|\nu(\pi)|} = (-1)^{mn - |\mu(\pi)|}$.
\end{prop}

\begin{proof} By Lemma~\ref{3_lem_macdonald_page_3},
\begin{multline*}
\left( \mu(\pi) + \lambda_{V(\pi)} + \rho_m \right) \cup \left( \nu(\pi)' + \lambda_{H(\pi)} + \rho_n\right) \\ = \left( \rho_{m + n} + \lambda \right)_{V(\pi)} \cup \left( \rho_{m + n} + \lambda \right)_{H(\pi)} \sorteq \lambda + \rho_{m + n},
\end{multline*} 
which implies that the map given in \eqref{3_eq_visual_interpretation_overlap_map} is well defined. We further see that the sign of the sorting permutation, say $\sigma$, only depends on the timing of the vertical steps of $\pi$. In fact, the identity permutation corresponds to the first $m$ steps of $\pi$ being vertical, and thus $\nu(\pi)$ being the empty partition. Removing a box from the Ferrers diagram of $\mu(\pi)$ and adding it to the diagram of $\nu(\pi)$ corresponds to composing $\sigma$ with a transposition, \textit{i.e.}\ multiplying its sign by $(-1)$.

We show that the map in \eqref{3_eq_visual_interpretation_overlap_map} is bijective by giving its inverse. Let $\mu$, $\nu$ be a pair of partitions whose $(m,n)$-overlap is $\lambda$. By definition, there exists a pair of subsequences $V$, $H \subset [m + n]$ with $V \cup_{m, n} H = [m + n]$ such that
$$\mu + \rho_m = \left( \lambda + \rho_{m + n} \right)_V \:\text{ and }\: \nu + \rho_n = \left( \lambda + \rho_{m + n} \right)_H
.$$
If $\pi \in \mathfrak{P}(n,m)$ denotes the staircase walk determined by $V(\pi) = V$ and $H(\pi) = H$, then Lemma~\ref{3_lem_macdonald_page_3} implies that $\mu = \mu(\pi) + \lambda_{V(\pi)}$ and $\nu = \nu(\pi)' + \lambda_{H(\pi)}$, which allows us to conclude that $\pi$ is the preimage of the pair $\mu$, $\nu$. 
\end{proof}

\begin{ex} \label{3_ex_visual_interpretation_overlap} Let us fix $m = 3$, $n = 6$ and a partition $\lambda = (7, 4, 3, 3, 3, 1)$ of length less than $m + n$. Proposition~\ref{3_prop_visual_interpretation_overlap} tells us that any staircase walk $\pi \in \mathfrak{P}(n,m)$ corresponds to a pair of partitions whose $(m,n)$-overlap equals $\lambda$. In order to visualize this correspondence, consider the following diagram of a staircase walk $\pi \in \mathfrak{P}(6,3)$ labeled by the partition $\lambda$:
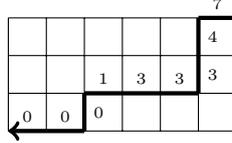
\vspace{-0.3cm}
\begin{center}
\begin{tikzpicture} 
\draw[step=0.5cm, thin] (0, 0) grid (3, 1.5);
\draw[ultra thick, ->] (1, 0) -- (0, 0);
\draw[ultra thick] (1, 0) -- (1, 0.5);
\draw[ultra thick] (1, 0.5) -- (2.5, 0.5);
\draw[ultra thick] (2.5, 0.5) -- (2.5, 1.5);
\draw[ultra thick] (2.5, 1.5) -- (3, 1.5);
\node[anchor=south] at (2.75, 1.5) {\tiny{7}};
\node[anchor=west] at (2.5, 1.25) {\tiny{4}};
\node[anchor=west] at (2.5, 0.75) {\tiny{3}};
\node[anchor=south] at (2.25, 0.5) {\tiny{3}};
\node[anchor=south] at (1.75, 0.5) {\tiny{3}};
\node[anchor=south] at (1.25, 0.5) {\tiny{1}};
\node[anchor=west] at (1, 0.25) {\tiny{0}};
\node[anchor=south] at (0.75, 0) {\tiny{0}};
\node[anchor=south] at (0.25, 0) {\tiny{0}};
\end{tikzpicture} 
\end{center}
Under the map defined in \eqref{3_eq_visual_interpretation_overlap_map}, $\pi$ is sent to the pair of partitions
\begin{multline*}
(\mu, \nu) = \left(\mu(\pi) + \lambda_{V(\pi)}, \nu(\pi)' + \lambda_{H(\pi)} \right) \\ = ((5, 5, 2) + (4, 3), (3, 1, 1, 1) + (7, 3, 3, 1)) = ((9, 8, 2), (10, 4, 4, 2)).
\end{multline*}
Recalling that the parts of $\mu(\pi)$ (or $\nu(\pi)'$) correspond to the rows of boxes above the staircase walk (or columns below the staircase walk), these numbers are easy to see in the diagram: for each part of $\mu(\pi)$ (or $\nu(\pi)'$), the label of the corresponding step of $\pi$ indicates the number of boxes that must be added to the row (or column) to obtain the corresponding part of $\mu$ (or $\nu$).
\end{ex}

The visualization of overlap explained in this example also provides a framework for visualizing pairs of partitions whose overlap is infinity. We recall that the $(m,n)$-overlap of two partitions, say $\mu$ and $\nu$, is infinity if and only if the sequence $\left( \mu + \rho_m \right) \cup \left( \nu + \rho_n \right)$ contains repetitions. The visualization of this case makes use of the notion of quasi-partitions.

\begin{defn} [quasi-partition] Let $\pi \in \mathfrak{P}(n,m)$. We call a sequence $\alpha$ of length $m + n$ a quasi-partition associated to the staircase walk $\pi$ if it satisfies all of the following conditions:
\begin{enumerate}
\item the elements of $\alpha$ are possibly negative integers but $\alpha_{m + n} \geq 0$;
\item there is no index $i$ so that $\alpha_{i - 1} < \alpha_i < \alpha_{i + 1}$;
\item if $i, i + 1 \in V(\pi)$ (or $i, i + 1 \in H(\pi)$), then $\alpha_{i + 1} \leq \alpha_i$;
\item if $i \in V(\pi)$ and $i + 1 \in H(\pi)$ (or vice versa), then $\alpha_{i + 1} \leq \alpha_i + 1$.
\end{enumerate}
\end{defn}
We remark that a sequence of length $m + n$ is a partition if and only if it is a quasi-partition associated to \emph{all} staircase walks in $\mathfrak{P}(n,m)$ -- unless $m = 1 = n$. In case $m = 1 = n$, a partition of length at most $2$ is still a quasi-partition associated to all staircase walks in $\mathfrak{P}(n,m)$, but the converse does not hold.

\begin{prop} \label{3_prop_visualization_of_infinite_overlap} Let $m$ and $n$ be non-negative integers. The $(m,n)$-overlap of two partitions $\mu$ and $\nu$ is equal to infinity if and only if there exist a staircase walk $\pi \in \mathfrak{P}(n,m)$ and a quasi-partition $\alpha$ associated to $\pi$ with the properties that $\alpha$ is \emph{not} a partition, $\mu = \mu(\pi) + \alpha_{V(\pi)}$ and $\nu = \nu(\pi)' + \alpha_{H(\pi)}$.
\end{prop}

\begin{ex*} Let us illustrate this visualization for partitions whose overlap is infinity on a concrete example. The first diagram depicts a staircase walk $\pi \in \mathfrak{P}(6,3)$ labeled by a quasi-partition $\alpha$ associated to $\pi$. 
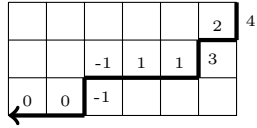
\begin{center}
\begin{tikzpicture} 
\draw[step=0.5cm, thin] (0, 0) grid (3, 1.5);
\draw[ultra thick, ->] (1, 0) -- (0, 0);
\draw[ultra thick] (1, 0) -- (1, 0.5);
\draw[ultra thick] (1, 0.5) -- (2.5, 0.5);
\draw[ultra thick] (2.5, 0.5) -- (2.5, 1);
\draw[ultra thick] (2.5, 1) -- (3, 1);
\draw[ultra thick] (3, 1) -- (3, 1.5);
\node[anchor=west] at (3, 1.25) {\tiny{4}};
\node[anchor=south] at (2.75, 1) {\tiny{2}};
\node[anchor=west] at (2.5, 0.75) {\tiny{3}};
\node[anchor=south] at (2.25, 0.5) {\tiny{1}};
\node[anchor=south] at (1.75, 0.5) {\tiny{1}};
\node[anchor=south] at (1.25, 0.5) {\tiny{-1}};
\node[anchor=west] at (1, 0.25) {\tiny{-1}};
\node[anchor=south] at (0.75, 0) {\tiny{0}};
\node[anchor=south] at (0.25, 0) {\tiny{0}};
\end{tikzpicture}
\end{center}
As in the preceding example, the label of each step of $\pi$ indicates the number of boxes that must be added to \emph{or removed from} the corresponding row of $\mu(\pi)$ (or column of $\nu(\pi)'$) to obtain the diagram of the partition $\mu$ (or $\nu$):
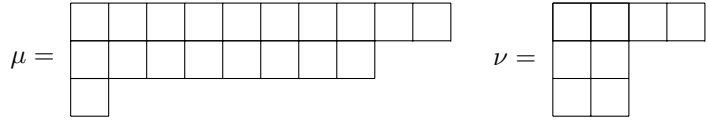
\begin{center}
\begin{tikzpicture}
\node(mu) at (-0.5, 0.75) {$\mu =$};
\draw[step=0.5cm, thin] (0, 0) grid (0.5, 0.5);
\draw[step=0.5cm, thin] (0, 0.5) grid (4, 1);
\draw[step=0.5cm, thin] (0, 0.99) grid (5, 1.5);
\end{tikzpicture}
\hspace{0.2cm}
\begin{tikzpicture}
\node(nu) at (-0.5, 0.75) {$\nu =$};
\draw[step=0.5cm, thin] (0, 0) grid (1, 1.5);
\draw[step=0.5cm, thin] (0, 0.99) grid (2, 1.5);
\end{tikzpicture}
\end{center}
Using the formal definition of overlap, we compute that 
$$\mu \star_{3,6} \nu = (10, 8, 1) \star_{3,6} (4,2,2) = \infty$$
since $8 + 1 = 4 + 5$.
\end{ex*}

\begin{proof}[Proof of Proposition~\ref{3_prop_visualization_of_infinite_overlap}] 
Fix a staircase walk $\pi \in \mathfrak{P}(n,m)$ and a quasi-partition $\alpha$ associated to $\pi$ with the property that $\alpha$ is not a partition. First we show that the sequence $\mu = \mu(\pi) + \alpha_{V(\pi)}$ is a partition. For any index $1 \leq p \leq m - 1$, $\mu_{p + 1} \leq \mu_p$: if $V(\pi)_{p + 1} = V(\pi)_p + 1$, the third property listed in the definition of quasi-partition implies that 
\begin{multline*}
\mu_{p + 1} = \mu(\pi)_{p + 1} + \alpha_{V(\pi)_{p + 1}} = \mu(\pi)_{p + 1} + \alpha_{V(\pi)_p + 1} \\ \leq \mu(\pi)_{p + 1} + \alpha_{V(\pi)_p}= \mu(\pi)_p + \alpha_{V(\pi)_p} = \mu_p
;
\end{multline*}
if $V(\pi)_{p + 1} = V(\pi)_p + q + 1$ for some $q \geq 2$, the combining the third and the fourth property implies that
\begin{multline*}
\mu_{p + 1} = \mu(\pi)_{p + 1} + \alpha_{V(\pi)_{p + 1}} = \mu(\pi)_{p + 1} + \alpha_{V(\pi)_p + q + 1} \\ \leq \mu(\pi)_{p + 1} + \alpha_{V(\pi)_p} + 2 = \mu(\pi)_p - q + \alpha_{V(\pi)_p} + 2 \leq \mu_p
;
\end{multline*}
if $V(\pi)_{p + 1} = V(\pi)_p + 2$, combining the second and the fourth property allows us to infer that
\begin{multline*}
\mu_{p + 1} = \mu(\pi)_{p + 1} + \alpha_{V(\pi)_{p + 1}} = \mu(\pi)_{p + 1} + \alpha_{V(\pi)_p + 2} \\ \leq \mu(\pi)_{p + 1} + \alpha_{V(\pi)_p} + 1 = \mu(\pi)_p - 1 + \alpha_{V(\pi)_p} + 1 = \mu_p
.
\end{multline*}
It also follows that $\mu_m \geq 0$: if $V(\pi)_m = m + n$, the first property implies that $$\mu_m = \mu(\pi)_{m} + \alpha_{V(\pi)_m} \geq \mu(\pi)_{m} = 0;$$
if $V(\pi)_m < m + n$, the first, third and fourth properties entail that $\alpha(V(\pi))_m \geq -1$ and thus that
$$\mu_m = \mu(\pi)_{m} + \alpha_{V(\pi)_m} \geq \mu(\pi)_{m} - 1 \geq 0.$$
We conclude that $\mu$ is indeed a partition. An analogous argument shows that the sequence $\nu = \nu(\pi)' + \alpha_{H(\pi)}$ is also a partition.

Second we show that the $(m,n)$-overlap of $\mu$ and $\nu$ is infinity by constructing a pair of indices $1 \leq p \leq m$ and $1 \leq q \leq n$ so that $\mu_p + m - p = \nu_q + n - q$. By assumption, $\alpha$ is not a partition, which a priori means that the quasi-partition $\alpha$ contains a strictly negative element or a strict increase. However, the condition that the last element of $\alpha$ be non-negative allows us to infer that $\alpha$ must contain a strict increase. More precisely, there exists an index $1 \leq i \leq m + n - 1$ so that $\alpha_{i + 1} = \alpha_i + 1$. According to the third and the fourth property, $i \in V(\pi)$ and $i + 1 \in H(\pi)$ (or vice versa). As the two cases are exact analogues, we may assume the first case. Let us introduce $p$ and $q$ by means of an annotated diagram of a possible staircase walk $\pi$:
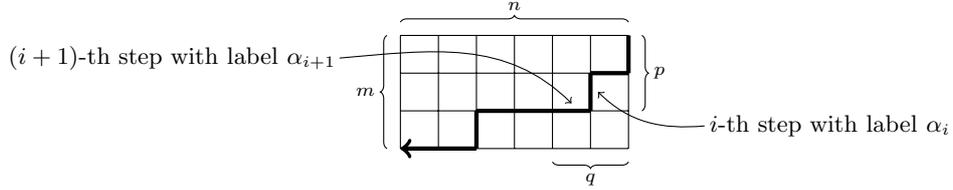
\begin{center}
\begin{tikzpicture} 
\draw[step=0.5cm, thin] (0, 0) grid (3, 1.5);
\draw[ultra thick, ->] (1, 0) -- (0, 0);
\draw[ultra thick] (1, 0) -- (1, 0.5);
\draw[ultra thick] (1, 0.5) -- (2.5, 0.5);
\draw[ultra thick] (2.5, 0.5) -- (2.5, 1);
\draw[ultra thick] (2.5, 1) -- (3, 1);
\draw[ultra thick] (3, 1) -- (3, 1.5);
\draw[decoration={brace, raise=5pt, mirror},decorate] (3, 0.5) -- node[right=6pt] {$\scriptstyle{p}$} (3, 1.5);
\draw[decoration={brace, raise=5pt, mirror},decorate] (0, 1.5) -- node[left=6pt] {$\scriptstyle{m}$} (0, 0);
\draw[decoration={brace, raise=5pt},decorate] (0, 1.5) -- node[above=6pt] {$\scriptstyle{n}$} (3, 1.5);
\draw[decoration={brace, raise=5pt},decorate] (3, 0) -- node[below=6pt] {$\scriptstyle{q}$} (2, 0);
\draw (4, 0.3) edge[out=185,in=-45,->] (2.6, 0.75);
\draw (-0.8, 1.2) edge[out=5,in=135,->] (2.25, 0.6);
\node at (-3, 1.2) {\small{$(i + 1)$-th step with label $\alpha_{i + 1}$}};
\node at (5.65, 0.3) {\small{$i$-th step with label $\alpha_i$}};
\end{tikzpicture}
\end{center}
\vspace{-0.3cm}
We see that
\begin{align*}
\mu_p + m - p ={} & \mu(\pi)_p + \alpha_{V(\pi)_p} + m - p = n - q + 1 + \alpha_i + m - p
\\
\nu_q + n - q ={} & \nu(\pi)'_q + \alpha_{H(\pi)_q} + n - q = m - p + \alpha_{i + 1} + n - q
.
\end{align*}
By construction, $\alpha_{i + 1} = \alpha_i + 1$, from which we conclude one direction of implication in the equivalence to be shown. 

In order to show the other direction, fix a pair of partitions, say $\mu$ and $\nu$, whose $(m,n)$-overlap is infinity. Let $\alpha$ be the sequence determined by the conditions that $\alpha + \rho_{m + n}$ be non-increasing and
\begin{align*}
\alpha + \rho_{m + n} \sorteq (\mu + \rho_m) \cup (\nu + \rho_n)
.
\end{align*}
Choose a pair of subsequences $V$, $H \subset [m + n]$ so that $\left( \alpha + \rho_{m + n} \right)_V = \mu + \rho_m$ and $\left( \alpha + \rho_{m + n} \right)_H = \nu + \rho_n$. It is worth noting that this choice is not unique since the sequences $\mu + \rho_m$ and $\nu + \rho_n$ have at least one element in common. Let $\pi \in \mathfrak{P}(n,m)$ be the staircase walk determined by $V(\pi) = V$ and $H(\pi) = H$. We claim that $\alpha$ is a quasi-partition associated to $\pi$ with the properties that $\alpha$ is \emph{not} a partition, $\mu = \mu(\pi) + \alpha_{V(\pi)}$ and $\nu = \nu(\pi)' + \alpha_{H(\pi)}$. The latter two follow immediately from Lemma~\ref{3_lem_macdonald_page_3}. In addition, the sequence $\alpha$ cannot be a partition since $\alpha + \rho_{m + n}$ is not strictly decreasing. The justification that $\alpha$ is indeed a quasi-partition is left to the reader. 
\end{proof}

The fact that the choice of $V$ and $H$ in the proof of Proposition~\ref{3_prop_visualization_of_infinite_overlap} is not unique entails that there is \emph{no} 1-to-1 correspondence between $\mathfrak{P}(n,m)$ and pairs of partitions whose $(m,n)$-overlap equals infinity.

\subsection{Complementary partitions and some of their properties}
In this section we study how taking the complement of a partition interacts with other operations on partitions, such as overlap. 

\begin{defn} [complement of a partition] The $(m,n)$-complement of a partition $\lambda$ contained in the rectangle $\left\langle m^n \right\rangle$ is given by \label{symbol_complement_of_partition}
\begin{align*}
\tilde \lambda = (m - \lambda_n, \dots, m - \lambda_1) \subset \langle m^n \rangle.
\end{align*}
When it is clear from the context with respect to which rectangle the complement is taken, we dispense with stating the parameters.
\end{defn}
The visual interpretation of this notion is that for any staircase walk $\pi \in \mathfrak{P}(m,n)$, the partitions $\mu(\pi)$ and $\nu(\pi)$ are $(m,n)$-complementary. Therefore, Proposition~\ref{3_prop_visual_interpretation_overlap} entails that two partitions are $(m,n)$-complementary if and only if they do not overlap, \textit{i.e.}\ if their $(n,m)$-overlap is the empty partition. In fact, this special case of Proposition~\ref{3_prop_visual_interpretation_overlap} is equivalent to Lemma~6 in \cite{dehaye12}.

In the following remark we collect a few properties of the complement. These observations are certainly not new given that they follow directly from the definition, but they will prove useful later.

\begin{rem} \label{3_rem_properties_of_complement} Taking the complement commutes with both conjugation and addition. More concretely, if $\lambda \subset \langle m^n \rangle$, then the $(n,m)$-complement of $\lambda'$ is conjugate to the $(m,n)$-complement of $\lambda$. For an additional partition $\kappa \subset \left\langle k^n \right\rangle$, the $(m + k,n)$-complement of $\lambda + \kappa$ is given by $\tilde\lambda + \tilde\kappa$. 
\end{rem}

Overlapping two partitions almost commutes with taking the complement. In fact, one could say that the two operations are skew-commutative.

\begin{lem} \label{3_lem_overlap_and_complement} If a partition $\lambda \subset \left\langle l^{m + n} \right\rangle$ is the $(m,n)$-overlap of the partitions $\mu$ and $\nu$, then $\tilde\lambda$ is the $(m,n)$-overlap of $\tilde\mu$ and $\tilde\nu$ where we view $\mu$ and $\nu$ as subsets of $\left\langle (n + l)^m \right\rangle$ and $\left\langle (m + l)^n \right\rangle$, respectively. Moreover, $\varepsilon_{m,n}(\tilde\mu, \tilde\nu) = (-1)^{mn} \varepsilon_{m,n}(\mu, \nu)$.
\end{lem}

\begin{proof} According to Proposition~\ref{3_prop_visual_interpretation_overlap}, whenever $\lambda = \mu \star_{m,n} \nu$ there exists a staircase walk $\pi \in \mathfrak{P}(n,m)$ such that $\mu = \mu(\pi) + \lambda_{V(\pi)}$ and $\nu = \nu(\pi)' + \lambda_{H(\pi)}$. Hence,
\begin{align*}
\tilde\mu = \nu(\pi) + \widetilde{\lambda_{V(\pi)}} \text{ and } \tilde\nu = \mu(\pi)' + \widetilde{\lambda_{H(\pi)}}
.
\end{align*}
If $\tau \in \mathfrak{P}(n,m)$ is the staircase walk obtained from $\pi$ by walking in the opposite direction (\textit{i.e.}\ up the stairs) and rotating the entire grid by 180 degrees, then  $\mu(\tau) = \nu(\pi)$ and $\nu(\tau) = \mu(\pi)$. Let us illustrate this relationship between the partitions associated to $\pi$ and $\tau$ by means of diagrams:
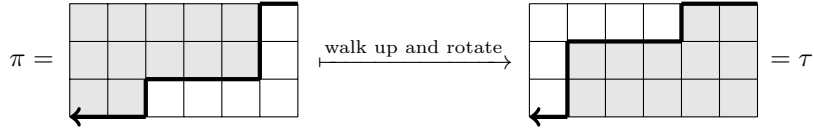
\begin{center}
\begin{tikzpicture} 
\node(pi) at (-0.5, 0.75) {$\pi =$};
\fill[black!10!white] (0, 0) rectangle (1, 1.5);
\fill[black!10!white] (1, 0.5) rectangle (2.5, 1.5);
\draw[step=0.5cm, thin] (0, 0) grid (3, 1.5);
\draw[ultra thick, ->] (1, 0) -- (0, 0);
\draw[ultra thick] (1, 0) -- (1, 0.5);
\draw[ultra thick] (1, 0.5) -- (2.5, 0.5);
\draw[ultra thick] (2.5, 0.5) -- (2.5, 1.5);
\draw[ultra thick] (2.5, 1.5) -- (3, 1.5);
\end{tikzpicture}
\begin{tikzpicture} 
\node(arrow) at (-1.5, 0.85) {$\xmapsto{\text{walk up and rotate}}$
};
\fill[black!10!white] (0.5, 0) rectangle (2, 1);
\fill[black!10!white] (2, 0) rectangle (3, 1.5);
\draw[step=0.5cm, thin] (0, 0) grid (3, 1.5);
\draw[ultra thick, ->] (0.5, 0) -- (0, 0);
\draw[ultra thick] (0.5, 0) -- (0.5, 1);
\draw[ultra thick] (0.5, 1) -- (2, 1);
\draw[ultra thick] (2, 1) -- (2, 1.5);
\draw[ultra thick] (2, 1.5) -- (3, 1.5);
\node(pi) at (3.45, 0.75) {$= \tau$};
\end{tikzpicture}
\end{center}
We see that the Ferrers diagrams of both $\mu(\pi)$ and $\nu(\tau)$ are determined by the boxes colored in gray, while the diagrams of both $\nu(\pi)$ and $\mu(\tau)$ correspond to the white boxes.
In addition, we infer from
\begin{align*}
V(\pi)_{m + 1 - i} = m + n + 1 - V(\tau)_i \:\text{ and }\: H(\pi)_{n + 1 - i} = m + n + 1 - H(\tau)_i
\end{align*}
that $\widetilde{\lambda_{V(\pi)}} = \tilde\lambda_{V(\tau)}$ and $\widetilde{\lambda_{H(\pi)}} = \tilde\lambda_{H(\tau)}$. Invoking again Proposition~\ref{3_prop_visual_interpretation_overlap}, we conclude that $\tilde{\lambda} = \tilde{\mu} \star_{m,n} \tilde{\nu}$.
The sign of this overlap is given by
\begin{align*}
\hspace{21.1pt} \varepsilon_{m,n}(\tilde\mu, \tilde\nu) = (-1)^{|\nu(\tau)|} = (-1)^{|\mu(\pi)|} = (-1)^{mn - |\nu(\pi)|} = (-1)^{mn} \varepsilon_{m,n}(\mu, \nu)
. \hspace{21.1pt} \qedhere 
\end{align*}
\end{proof}

\subsection{Marked staircase walks}
Our second visualization for the overlap of two partitions makes use of the notion of subpartitions, which are obtained by viewing overlap as a containment relation on partitions.

\begin{defn} [subpartition] Let $\lambda$ and $\mu$ be partitions. We call $\mu$ an $(m,n)$-subpartition of $\lambda$ if there exists a partition $\nu$ such that
$\mu \star_{m,n} \nu = \lambda$. Equivalently, $\mu$ is an $(m,n)$-subpartition of $\lambda$ if there exists a subsequence $M \subset [m + n]$ with $l(M) = m$ such that for all $1 \leq j \leq m$
\begin{align*}
\mu_j + m - j = \lambda_{M_j} + m + n - M_j.
\end{align*}
and $\mu_{m + 1} = 0 = \lambda_{m + n + 1}$. We denote the subpartition of $\lambda$ corresponding to the subsequence $M \subset [m + n]$ by \label{symbol_subpartition} $\sub_{m + n}(\lambda, M)$. If the parameter $m + n$ is clear from the context, it is sometimes omitted.
\end{defn}

\begin{ex} Let $\lambda$ be a partition of length at most $n$. The easiest example of a subpartition of $\lambda$ corresponds to removing one element from $[n]$: for $1 \leq j \leq n$, \begin{align} \label{3_ex_easiest_subpartition_eq}
\sub_n\left(\lambda, [n] \setminus (j) \right) = \lambda_{[n] \setminus (j)} + \left\langle 1^{j - 1} \right\rangle.
\end{align} 
This observation makes it possible to construct subpartitions $\sub(\lambda, M)$ for $M \subset [n]$ iteratively.
\end{ex}

In order to formally state our second visual interpretation for overlap, we also require the following technical definition.

\begin{defn} Let $n$ be a non-negative integer and $K \subset [n]$ a subsequence. We define \label{symbol_C_n(K)}
\begin{align*}
C_n(K) \sorteq (n - j + 1: j \not\in K)
\end{align*}
so that $C_n(K)$ is a subsequence of $[n]$.
\end{defn}
Let us give a quick numerical example: $C_6((1,2,4,5)) = (1,4)$.

\begin{lem} \label{3_lem_overlap_with_subpartition} Let $\lambda \subset \langle m^n \rangle$ be a partition and $K \subset [n]$ a subsequence. If $\kappa$ is the subpartition $\sub_n(\lambda, K)$, then
\begin{align} \label{3_lem_overlap_with_subpartition_eq}
\lambda' \star_{m, l(C_n(K))} \sub\left(\tilde\lambda, C_n(K) \right) = \kappa'
.
\end{align}
Furthermore, $\varepsilon_{m, l(C_n(K))}\left(\lambda', \sub_n\left(\tilde\lambda, C_n(K) \right) \right) = (-1)^{\left|\tilde\lambda_{C_n(K)}\right|}$.
\end{lem}

\begin{proof} We prove a slightly stronger statement by induction on the length of $K$. We claim that for each $K \subset [n]$, there exists a staircase walk $\pi \in \mathfrak{P}(l(C_n(K)), m)$ with the following properties:
\begin{enumerate}
\item $\lambda' = \mu(\pi) + \left( \sub(\lambda, K)' \right)_{V(\pi)}$;
\item $\sub\left(\tilde\lambda, C_n(K) \right) = \nu(\pi)' + \left( \sub(\lambda, K)' \right)_{H(\pi)}$;
\item for each element $i \in [n]$ with $i < \min\{K\}$, the $(l(C_n(K)) - i + 1)$-th \emph{horizontal} step of $\pi$ is the $(\lambda_i + l(C_n(K)) - i + 1)$-th step of $\pi$.
\end{enumerate}
Notice that the existence of a staircase walk $\pi$ with the first two properties is equivalent to the equality stated in \eqref{3_lem_overlap_with_subpartition_eq}, according to the correspondence given in Proposition~\ref{3_prop_visual_interpretation_overlap}. For the base case $l(K) = 0$, we put a visual interpretation on the fact that $\lambda'$ and $\tilde\lambda'$ are $(n,m)$-complementary to infer the existence of $\pi \in \mathfrak{P}(n,m)$ such that $\lambda' = \mu(\pi)$ and $\tilde\lambda = \nu(\pi)'$. By definition, $\pi$ satisfies the first two conditions, and the third can be read off the following annotated diagram of the staircase walk $\pi \in \mathfrak{P}(6,3)$ with $\mu(\pi) = \lambda' = (5, 5, 2)$ (colored in gray):
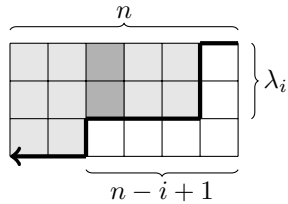
\begin{center}
\begin{tikzpicture} 
\fill[black!10!white] (0, 0) rectangle (1, 1.5);
\fill[black!10!white] (1, 0.5) rectangle (2.5, 1.5);
\fill[black!30!white] (1, 0.5) rectangle (1.5, 1.5);
\draw[step=0.5cm, thin] (0, 0) grid (3, 1.5);
\draw[ultra thick, ->] (1, 0) -- (0, 0);
\draw[ultra thick] (1, 0) -- (1, 0.5);
\draw[ultra thick] (1, 0.5) -- (2.5, 0.5);
\draw[ultra thick] (2.5, 0.5) -- (2.5, 1.5);
\draw[ultra thick] (2.5, 1.5) -- (3, 1.5);
\draw[decoration={brace, raise=5pt},decorate] (3, 1.5) -- node[right=6pt] {$\lambda_i$} (3, 0.5);
\draw[decoration={brace, raise=5pt, mirror},decorate] (1, 0) -- node[below=6pt] {$n - i + 1$} (3, 0);
\draw[decoration={brace, raise=5pt},decorate] (0, 1.5) -- node[above=6pt] {$n$} (3, 1.5);
\end{tikzpicture}
\end{center}
The annotations are based on the case $i = 3$.

For the induction step, consider a subsequence $(k) \cup K \subset [n]$ together with a staircase walk $\pi \in \mathfrak{P}(l(C_n(K)), m)$ which possesses the three properties stated above for the subsequence $K$. Construct a staircase walk $\tau \in \mathfrak{P}(l(C_n(K)) - 1, m)$ by removing the $(l(C_n(K)) - k + 1)$-th horizontal step from $\pi$; or equivalently, by removing the $(\lambda_k + l(C_n(K)) - k + 1)$-th step from $\pi$. By construction, the third property still holds for all elements $i \in [n]$ with $i < \min \{ (k) \cup K \} = k$. In order to justify that $\tau$ also satisfies the other two conditions, we first observe that 
\begin{multline} \label{3_in_proof_lem_overlap_with_subpartition}
\sub(\lambda, (k) \cup K)' = \left( (\lambda_k + l(C_n(K)) - k) \cup \sub(\lambda, K) \right)' \\ = \left\langle 1^{\lambda_k + l(C_n(K)) - k} \right\rangle + \sub(\lambda, K)'
.
\end{multline}
In particular, $\sub(\lambda, K)'$ has length at most $\lambda_k + l(C_n(K)) - k$. Hence, the fact that the first $\lambda_k + l(C_n(K)) - k$ steps of $\pi$ and $\tau$ are identical allows us to deduce that
\begin{align*}
\nu(\tau)' + \left(\sub(\lambda, (k) \cup K)'\right)_{H(\tau)} ={} & \nu(\tau)'+ \left(\sub(\lambda, K)'\right)_{H(\tau)} + \left\langle 1^{\lambda_k + l(C_n(K)) - k} \right\rangle_{H(\tau)} 
\\
={} & \nu(\tau)' + \left(\sub(\lambda, K)'\right)_{H(\pi)} + \left\langle 1^{l(C_n(K)) - k} \right\rangle
.
\intertext{By construction, $\nu(\tau)'$ is a subsequence of $\nu(\pi)'$:}
\nu(\tau)' + \left(\sub(\lambda, (k) \cup K)'\right)_{H(\tau)} ={} & \left( \nu(\pi)' \right)_{[l(C_n(K))] \setminus (l(C_n(K)) - k + 1)} \\ & + \left(\sub(\lambda, K)'\right)_{H(\pi)} + \left\langle 1^{l(C_n(K)) - k} \right\rangle
.
\intertext{Hence, the second property of $\pi$, together with the observation that the length of $\left(\sub(\lambda, K)'\right)_{H(\pi)}$ is at most $l(C_n(K)) - k$, gives}
\nu(\tau)' + \left(\sub(\lambda, (k) \cup K)'\right)_{H(\tau)} ={} & \sub\left(\tilde{\lambda}, C_n(K)\right)_{[l(C_n(K))] \setminus (l(C_n(K)) - k + 1)} \\ & + \left\langle 1^{l(C_n(K)) - k} \right\rangle
.
\intertext{Finally, the equality in \eqref{3_ex_easiest_subpartition_eq} on page \pageref{3_ex_easiest_subpartition_eq} states that}
\nu(\tau)' + \left(\sub(\lambda, (k) \cup K)'\right)_{H(\tau)} ={} & \sub\left(\tilde{\lambda}(C_n(K)), [l(C_n(K))] \setminus (l(C_n(K)) - k + 1)\right) 
\\
={} & \sub\left(\tilde{\lambda}, C_n(K)_{[l(C_n(K))] \setminus (l(C_n(K)) - k + 1)} \right) 
\\
={} & \sub \left( \tilde{\lambda}, C_n(K) \setminus (n - k + 1)\right) \\
={} & \sub \left( \tilde{\lambda}, C((k) \cup K) \right)
.
\end{align*}
Combining the first and the third property of $\pi$ allows us to infer the first property for $\tau$. Indeed, by the equality given in \eqref{3_in_proof_lem_overlap_with_subpartition},
\begin{align*}
\mu(\tau) + \left(\sub(\lambda, (k) \cup K)'\right)_{V(\tau)} ={} & \mu(\tau) + \left\langle 1^{\lambda_k + l(C_n(K)) - k} \right\rangle_{V(\tau)} + \left(\sub(\lambda, K)'\right)_{V(\tau)}
.
\intertext{Given that $\sub(\lambda, K)'$ is at most of length $\lambda_k + l(C_n(K)) - k$, the observation that the first $\lambda_k + l(C_n(K)) - k$ steps of $\pi$ and $\tau$ are identical allows us to infer that}
\mu(\tau) + \left(\sub(\lambda, (k) \cup K)'\right)_{V(\tau)} ={} & \mu(\tau) + \left\langle 1^{\lambda_k + l(C_n(K)) - k} \right\rangle_{V(\pi)} + \left(\sub(\lambda, K)'\right)_{V(\pi)}
\intertext{Moreover, the third property implies that there are exactly $\lambda_k$ vertical steps among the first $\lambda_k + l(C_n(K)) - k$ steps of $\pi$:}
\mu(\tau) + \left(\sub(\lambda, (k) \cup K)'\right)_{V(\tau)} ={} & \mu(\tau) + \left\langle 1^{\lambda_k} \right\rangle + \left(\sub(\lambda, K)'\right)_{V(\pi)}
.
\intertext{By construction, $\mu(\tau) = \mu(\pi) - \left\langle 1^{\lambda_k} \right\rangle$, from which we conclude that}
\mu(\tau) + \left(\sub(\lambda, (k) \cup K)'\right)_{V(\tau)} ={} & \mu(\pi) + \left(\sub(\lambda, K)'\right)_{V(\pi)} = \lambda'
.
\end{align*} 
This justifies the claim, and thus the equality in \eqref{3_lem_overlap_with_subpartition_eq}. For the statement on the sign, recall that Proposition ~\ref{3_prop_visual_interpretation_overlap} entails that 
\begin{align*}
\varepsilon_{m, l(C_n(K))}\left(\lambda', \sub\left(\tilde\lambda, C_n(K)\right)\right) = (-1)^{|\nu(\pi)|} = (-1)^{\left|\nu(\pi)' \right|} = (-1)^{\left| {\tilde{\lambda}}_{C_n(K)} \right|}
\end{align*}
since $\nu(\pi)' = \tilde{\lambda}_{C_n(K)}$ by construction.
\end{proof}

\begin{ex} \label{3_ex_visualize_construction_in_lem_overlap_with_subpartition} To visualize this construction on a concrete example, fix $m = 4$, $n = 7$, a partition $\lambda = (4,4,2,2,1,1,1) \subset \left\langle 4^7 \right\rangle$ and $K = (1, 4, 5, 7) \subset [7]$. Draw the diagram of the staircase walk $\pi \in \mathfrak{P}(7,4)$ that is determined by the condition that $\mu(\pi) = \lambda'$, and then mark/color the horizontal steps that lie in the subsequence $H(\pi)_{[7] \setminus C_7(K)}$. In our example $C_7(K) = (2,5,6)$. 
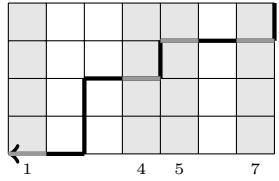
\begin{center}
\begin{tikzpicture}
\fill[black!10!white] (3, 0) rectangle (3.5, 2);
\fill[black!10!white] (2, 0) rectangle (2.5, 2);
\fill[black!10!white] (1.5, 0) rectangle (2, 2);
\fill[black!10!white] (0, 0) rectangle (0.5, 2);
\draw[step=0.5cm, thin] (0, 0) grid (3.5, 2);
\draw[ultra thick] (3.5, 2) -- (3.5, 1.5);
\draw[ultra thick] (3.5, 1.5) -- (2, 1.5);
\draw[ultra thick] (2, 1.5) -- (2, 1);
\draw[ultra thick] (2, 1) -- (1, 1);
\draw[ultra thick] (1, 1) -- (1, 0);
\draw[ultra thick, ->] (1, 0) -- (0, 0);
\draw[ultra thick, black!40!white] (3.5, 1.5) -- (3, 1.5);
\draw[ultra thick, black!40!white] (2.5, 1.5) -- (2, 1.5);
\draw[ultra thick, black!40!white] (2, 1) -- (1.5, 1);
\draw[ultra thick, black!40!white] (0.5, 0) -- (0, 0);
\node[anchor=north] at (3.25, 0) {\tiny{7}};
\node[anchor=north] at (2.25, 0) {\tiny{5}};
\node[anchor=north] at (1.75, 0) {\tiny{4}};
\node[anchor=north] at (0.25, 0) {\tiny{1}};
\end{tikzpicture}
\end{center}
The numbers along the bottom indicate which elements of $K$ the marked horizontal steps correspond to. Now, imagine that $\pi$ is labeled by the empty partition $(0, \dots, 0)$ as described in Example \ref{3_ex_visual_interpretation_overlap}. Throughout this construction the labeling of the staircase walk will keep track of $\sub(\lambda, L)'$ where $L$ consists of the elements of $K$ that have been removed from the diagram. At the moment, $\sub(\lambda, L)' = \sub(\lambda, \emptyset)' = \emptyset$, which matches the (imaginary) labels.

Following the construction outlined in the proof of Lemma~\ref{3_lem_overlap_with_subpartition}, remove the marked horizontal step corresponding to the largest element $k \in K$ and increase the label of each preceding step by 1:
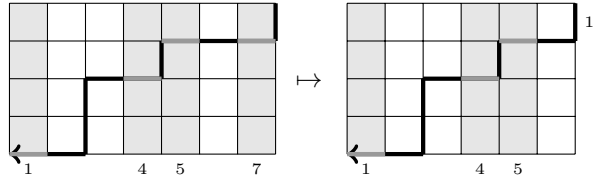
\begin{center}
\begin{tikzpicture}
\fill[black!10!white] (3, 0) rectangle (3.5, 2);
\fill[black!10!white] (2, 0) rectangle (2.5, 2);
\fill[black!10!white] (1.5, 0) rectangle (2, 2);
\fill[black!10!white] (0, 0) rectangle (0.5, 2);
\draw[step=0.5cm, thin] (0, 0) grid (3.5, 2);
\draw[ultra thick] (3.5, 2) -- (3.5, 1.5);
\draw[ultra thick] (3.5, 1.5) -- (2, 1.5);
\draw[ultra thick] (2, 1.5) -- (2, 1);
\draw[ultra thick] (2, 1) -- (1, 1);
\draw[ultra thick] (1, 1) -- (1, 0);
\draw[ultra thick, ->] (1, 0) -- (0, 0);
\draw[ultra thick, black!40!white] (3.5, 1.5) -- (3, 1.5);
\draw[ultra thick, black!40!white] (2.5, 1.5) -- (2, 1.5);
\draw[ultra thick, black!40!white] (2, 1) -- (1.5, 1);
\draw[ultra thick, black!40!white] (0.5, 0) -- (0, 0);
\node[anchor=north] at (3.25, 0) {\tiny{7}};
\node[anchor=north] at (2.25, 0) {\tiny{5}};
\node[anchor=north] at (1.75, 0) {\tiny{4}};
\node[anchor=north] at (0.25, 0) {\tiny{1}};
\end{tikzpicture}
\begin{tikzpicture}
\node(arrow) at (-0.5, 0.95) {$\mapsto$};
\fill[black!10!white] (2, 0) rectangle (2.5, 2);
\fill[black!10!white] (1.5, 0) rectangle (2, 2);
\fill[black!10!white] (0, 0) rectangle (0.5, 2);
\draw[step=0.5cm, thin] (0, 0) grid (3, 2);
\draw[ultra thick] (3, 2) -- (3, 1.5);
\draw[ultra thick] (3, 1.5) -- (2, 1.5);
\draw[ultra thick] (2, 1.5) -- (2, 1);
\draw[ultra thick] (2, 1) -- (1, 1);
\draw[ultra thick] (1, 1) -- (1, 0);
\draw[ultra thick, ->] (1, 0) -- (0, 0);
\draw[ultra thick, black!40!white] (2.5, 1.5) -- (2, 1.5);
\draw[ultra thick, black!40!white] (2, 1) -- (1.5, 1);
\draw[ultra thick, black!40!white] (0.5, 0) -- (0, 0);
\node[anchor=north] at (2.25, 0) {\tiny{5}};
\node[anchor=north] at (1.75, 0) {\tiny{4}};
\node[anchor=north] at (0.25, 0) {\tiny{1}};
\node[anchor=west] at (3, 1.75) {\tiny{1}};
\end{tikzpicture}
\end{center} 
According to the equality in \eqref{3_in_proof_lem_overlap_with_subpartition}, we have that
$$\sub(\lambda, L)' = \sub(\lambda, (7))' = \left\langle 1^{\lambda_7 + 7 - 7} \right\rangle + \emptyset = (1),$$
which thus matches the labels. In fact, the number of steps preceding the marked horizontal step corresponding to the largest remaining element $k \in K$ is always equal to $\lambda_k + l(C(L)) - k$, owing to the third property shown in the proof of the preceding Lemma. Therefore, the recursive equality in \eqref{3_in_proof_lem_overlap_with_subpartition} entails that increasing the labels of the preceding steps by 1 upon the removal of an element $k$ from the diagram ensures that the labels of the staircase walk will always match $\sub(\lambda, L)'$. Proceeding in this manner, you thus end up with a staircase walk labeled by $\sub(\lambda, K)'$:
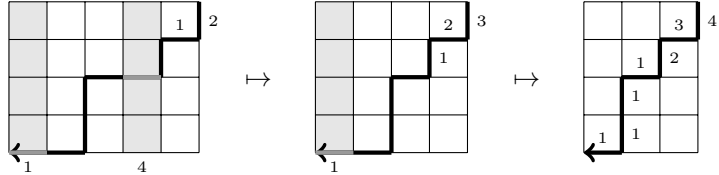
\begin{center}
\begin{tikzpicture}
\fill[black!10!white] (1.5, 0) rectangle (2, 2);
\fill[black!10!white] (0, 0) rectangle (0.5, 2);
\draw[step=0.5cm, thin] (0, 0) grid (2.5, 2);
\draw[ultra thick] (2.5, 2) -- (2.5, 1.5);
\draw[ultra thick] (2.5, 1.5) -- (2, 1.5);
\draw[ultra thick] (2, 1.5) -- (2, 1);
\draw[ultra thick] (2, 1) -- (1, 1);
\draw[ultra thick] (1, 1) -- (1, 0);
\draw[ultra thick, ->] (1, 0) -- (0, 0);
\draw[ultra thick, black!40!white] (2, 1) -- (1.5, 1);
\draw[ultra thick, black!40!white] (0.5, 0) -- (0, 0);
\node[anchor=west] at (2.5, 1.75) {\tiny{2}};
\node[anchor=south] at (2.25, 1.5) {\tiny{1}};
\node[anchor=north] at (1.75, 0) {\tiny{4}};
\node[anchor=north] at (0.25, 0) {\tiny{1}};
\end{tikzpicture}
\begin{tikzpicture}
\node(arrow) at (-0.75, 0.95) {$\mapsto$};
\fill[black!10!white] (0, 0) rectangle (0.5, 2);
\draw[step=0.5cm, thin] (0, 0) grid (2, 2);
\draw[ultra thick] (2, 2) -- (2, 1.5);
\draw[ultra thick] (2, 1.5) -- (1.5, 1.5);
\draw[ultra thick] (1.5, 1.5) -- (1.5, 1);
\draw[ultra thick] (1.5, 1) -- (1, 1);
\draw[ultra thick] (1, 1) -- (1, 0);
\draw[ultra thick, ->] (1, 0) -- (0, 0);
\draw[ultra thick, black!40!white] (0.5, 0) -- (0, 0);
\node[anchor=west] at (2, 1.75) {\tiny{3}};
\node[anchor=south] at (1.75, 1.5) {\tiny{2}};
\node[anchor=west] at (1.5, 1.25) {\tiny{1}};
\node[anchor=north] at (0.25, 0) {\tiny{1}};
\end{tikzpicture}
\begin{tikzpicture}
\node(arrow) at (-0.75, 0.95) {$\mapsto$};
\draw[step=0.5cm, thin] (0, 0) grid (1.5, 2);
\draw[ultra thick] (1.5, 2) -- (1.5, 1.5);
\draw[ultra thick] (1.5, 1.5) -- (1, 1.5);
\draw[ultra thick] (1, 1.5) -- (1, 1);
\draw[ultra thick] (1, 1) -- (0.5, 1);
\draw[ultra thick] (0.5, 1) -- (0.5, 0);
\draw[ultra thick, ->] (0.5, 0) -- (0, 0);
\node[anchor=south] at (0.25, 0) {\tiny{1}};
\node[anchor=west] at (0.5, 0.25) {\tiny{1}};
\node[anchor=west] at (0.5, 0.75) {\tiny{1}};
\node[anchor=south] at (0.75, 1) {\tiny{1}};
\node[anchor=west] at (1, 1.25) {\tiny{2}};
\node[anchor=south] at (1.25, 1.5) {\tiny{3}};
\node[anchor=west] at (1.5, 1.75) {\tiny{4}};
\node[anchor=north] at (0.25, 0) {\color{white}{\tiny{1}}};
\end{tikzpicture}
\end{center}
In fact, the last diagram is our visualization for the correspondence between staircase walks in a $3 \times 4$-rectangle and pairs of partitions whose $(4,3)$-overlap is equal to $\sub(\lambda, K)' = (4, 3, 2, 1, 1, 1, 1)$. According to Lemma~\ref{3_lem_overlap_with_subpartition}, the pair of partitions corresponding to this diagram is $\left( \lambda', \sub\left( \tilde{\lambda}, C(K)\right) \right)$. The following proposition states that this algorithmic procedure is invertible.
\end{ex}

\begin{prop} \label{3_prop_subpartition_interpretation_overlap} For a fixed partition $\kappa \subset \left\langle (m + n)^l \right\rangle$, there is a 1-to-1 correspondence between 
\begin{enumerate}
\item $\{(\mu, \nu): \mu \star_{m,n} \nu = \kappa'\}$
\item $\{(\lambda, K): \lambda \subset \left\langle m^{n + l} \right\rangle, K \subset [n + l], l(K) = l \text{ and } \sub_{n + l}(\lambda, K) = \kappa\}$
\end{enumerate}
given by the following mapping:
\begin{align} \label{3_prop_subpartition_interpretation_overlap_map}
(\lambda, K) \mapsto \left(\lambda', \sub_{n + l} \left(\tilde\lambda, C_{n + l}(K)\right)\right).
\end{align} 
\end{prop}

\begin{proof} By Lemma~\ref{3_lem_overlap_with_subpartition}, the map in \eqref{3_prop_subpartition_interpretation_overlap_map} is well defined. We observe that it is also injective. Indeed, if $(\lambda, K)$ and $(\eta, L)$ are both mapped to $(\mu, \nu)$, then $\lambda = \mu' = \eta$ and hence $$\left( \tilde\lambda + \rho_{n + l}\right)_{C_{n + l}(K)} = \nu + \rho_n = \left( \tilde\lambda + \rho_{n + l}\right)_{C_{n + l}(L)},$$ which entails that $K = L$ because the three sequences in question are strictly decreasing. Rather than showing directly that the map in \eqref{3_prop_subpartition_interpretation_overlap_map} is surjective, we will prove that both of the sets stated in the proposition are of cardinality $\binom{m + n}{m}$. For the first set, this is an immediate consequence of Proposition~\ref{3_prop_visual_interpretation_overlap}. 

For the second set, we prove the claim by induction on the length of the partition $\kappa$. Let us denote the second set by $\calS_2(\kappa)$. For the base case, we compute its cardinality under the assumption that $\kappa$ is the empty partition. In this case, any pair $(\lambda, K)$ that lies in $\calS_2(\kappa)$ satisfies $\rho_l = \left( \lambda + \rho_{n + l} \right)_K$. Given that the latter sequence is strictly decreasing, this equality implies that $K = (n + 1, \dots, n + l)$ and $\lambda_K = \emptyset$. Hence, $\lambda$ ranges over all partitions that are contained in the rectangle $\langle m^n \rangle$, of which there are exactly $\binom{m + n}{m}$. 

For the induction step, consider some partition $\kappa$ of length $0 < i \leq l$. We will construct a pair $(\lambda, K) \in \calS_2(\kappa)$ from each pair $(\eta, L) \in \calS_2((\kappa_1, \dots, \kappa_{i - 1}, 0))$. Let $j$ be the largest index so that $\eta_{j - 1} + n + l - (j - 1) > \kappa_i + l - i$ (where we use the convention that $\eta_0$ is infinitely large). By definition, $L_{i - 1} < j \leq L_i$. Indeed,
\begin{align*}
\eta_{L_{i - 1}} + n + l - L_{i - 1} ={} & \kappa_{i - 1} + l - (i - 1) > \kappa_i + l - i
\intertext{and}
\eta_{L_i} + n + l - L_i ={} & 0 + l - i \leq \kappa_i + l - i
.
\end{align*}
Hence, $K = (L_1, \dots, L_{i - 1}, j, L_{i + 1}, \dots, L_l)$ defines a subsequence of $[n + l]$. Moreover, the definition of $j$ ensures that $\lambda = (\eta_1, \dots, \eta_{j - 1}, \kappa_i - i - n + j, \eta_j, \dots, \eta_{n + l - 1})$ is a partition. It is left to the reader to verify that the pair $(\lambda, K)$ is an element of $\calS_2(\kappa)$. The induction hypothesis thus allows us to conclude that the cardinality of the set $\calS_2(\kappa)$ is at least $\binom{m + n}{m}$. Recalling that the map in \eqref{3_prop_subpartition_interpretation_overlap_map} is an injection from $\calS_2(\kappa)$ to a set of cardinality $\binom{m + n}{m}$ completes the proof.
\end{proof}

Keeping in mind Example~\ref{3_ex_visualize_construction_in_lem_overlap_with_subpartition}, it is easy to give a visual description of the inverse of the map defined in \eqref{3_prop_subpartition_interpretation_overlap_map}: Given a pair of partitions $\mu$ and $\nu$ whose $(m,n)$-overlap equals $\kappa'$, use the visual interpretation described in Example~\ref{3_ex_visual_interpretation_overlap} to associate it to a staircase walk $\pi \in \mathfrak{P}(n,m)$ labeled by $\kappa'$. Iteratively construct a staircase walk $\tau \in \mathfrak{P}(n + l, m)$ by inserting a \emph{marked} horizontal step between any two adjacent steps of $\pi$ with distinct labels and simultaneously decreasing by 1 the labels of all steps to the right of the insertion. Here we use the convention that the ``label'' before the first step is $l$, while the ``label'' after the last step is 0. As illustrated in Example~\ref{3_ex_visualize_construction_in_lem_overlap_with_subpartition}, we can then map the pair $(\mu, \nu)$ to the partition $\lambda = \mu(\tau)'$ and the following subsequence $K \subset [n]$: $K \sorteq (n + l - i + 1: H(\pi)_i \text{ is marked})$.

\section{More overlap identities} \label{3_section_more_overlap_identities}
Applying the different ways of seeing the overlap of two partitions, which we discussed in the preceding section, to the second overlap identity allows us to derive more overlap identities. This will allow us to regard the dual Cauchy identity as an overlap identity.

\subsection{Variations on the second overlap identity}
\begin{cor} 
Let $0 \leq l \leq \min\{n - k, n\}$. Let $\calS$, $\calT$ and $\calY$ be sets containing $l$, $n - l$ and $m$ variables, respectively, so that $\Delta(\calY) \neq 0$ and $\Delta(\calS; \calT) \neq 0$. Suppose that $k$ is the $(m,n)$-index of a partition $\lambda$, then
\begin{align} \label{3_cor_first_visualization_of_overlap_LS_eq}
\begin{split} 
LS_{\lambda} (-(\calS \cup \calT); \calY) 
={} & \sum_{\pi \in \mathfrak{P}(m + n - k - l, l)} (-1)^{|\nu(\pi_1)|} \frac{\Delta\left(\calY_{H(\pi_2)}; \calS \right) \Delta\left(\calT; \calY_{V(\pi_2)}\right)}{\Delta \left(\calY_{H(\pi_2)}; \calY_{V(\pi_2)}\right) \Delta(\calT; \calS)} \\
& \times LS_{\mu(\pi_1) + \lambda_{V(\pi_1)} - \left\langle (m - k)^{l(V(\pi_1))} \right\rangle} \left(-\calS; \calY_{V(\pi_2)}\right) \\
& \times LS_{\nu(\pi_1)' + \lambda_{H(\pi_1)} \cup \lambda_{(n + 1 - k, n + 2 - k, \dots )}} \left(-\calT; \calY_{H(\pi_2)}\right)
\end{split}
\end{align}
where $\pi_1$ denotes the $n - k$ first steps of $\pi$, while $\pi_2$ denotes the $m$ last steps of $\pi$. We view $\pi_1$ and $\pi_2$ as staircase walks inside the appropriate rectangles.
\end{cor}

The definitions of the ``partial'' staircase walks $\pi_1$ and $\pi_2$ are best explained by means of a diagram: let $m = 5$, $n = 8$, $k = 4$ and $l = 3$, then the following staircase walk $\pi \in \mathfrak{P}(m + n - k - l, l)$ splits into $\pi_1 \in \mathfrak{P}(2,2)$ and $\pi_2 \in \mathfrak{P}(4,1)$ comprising of $n - k$ and $m$ steps, respectively.
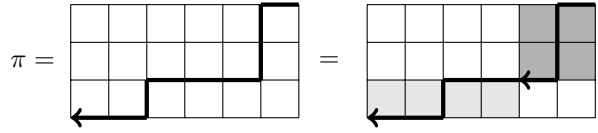
\begin{center}
\begin{tikzpicture} 
\node(pi) at (-0.5, 0.75) {$\pi =$};
\draw[step=0.5cm, thin] (0, 0) grid (3, 1.5);
\draw[ultra thick, ->] (1, 0) -- (0, 0);
\draw[ultra thick] (1, 0) -- (1, 0.5);
\draw[ultra thick] (1, 0.5) -- (2.5, 0.5);
\draw[ultra thick] (2.5, 0.5) -- (2.5, 1.5);
\draw[ultra thick] (2.5, 1.5) -- (3, 1.5);
\end{tikzpicture}
\begin{tikzpicture} 
\node(pi) at (-0.5, 0.75) {$=$};
\fill[black!10!white] (0, 0) rectangle (2, 0.5);
\fill[black!30!white] (2, 0.5) rectangle (3, 1.5);
\draw[step=0.5cm, thin] (0, 0) grid (3, 1.5);
\draw[ultra thick, ->] (1, 0) -- (0, 0);
\draw[ultra thick] (1, 0) -- (1, 0.5);
\draw[ultra thick] (1, 0.5) -- (2.5, 0.5);
\draw[ultra thick] (2.5, 0.5) -- (2.5, 1.5);
\draw[ultra thick] (2.5, 1.5) -- (3, 1.5);
\draw[ultra thick, ->] (2.5, 0.5) -- (2, 0.5);
\end{tikzpicture}
\end{center}
The diagrams of $\pi_1$ and $\pi_2$ are colored in different shades of gray.

\begin{proof} The right-hand side of the equation in \eqref{3_cor_first_visualization_of_overlap_LS_eq} is equal to
\begin{align*}
\RHS ={} & \sum_{p = 0}^{\min\{l,m\}} \sum_{\pi_2 \in \mathfrak{P}(m - p, p)} \hspace{10pt} \sum_{\substack{\pi_1 \in \mathfrak{P}(n - k - l + p, l - p)}} \frac{\Delta\left(\calY_{H(\pi_2)}; \calS \right) \Delta\left(\calT; \calY_{V(\pi_2)}\right)}{\Delta \left(\calY_{H(\pi_2)}; \calY_{V(\pi_2)}\right) \Delta(\calT; \calS)} \\
& \times (-1)^{|\nu(\pi_1)|} LS_{\mu(\pi_1) + \lambda_{V(\pi_1)} - \left\langle (m - k)^{l(V(\pi_1))} \right\rangle} \left(-\calS; \calY_{V(\pi_2)}\right) \\
& \times LS_{\nu(\pi_1)' + \lambda_{H(\pi_1)} \cup \lambda_{(n + 1 - k, n + 2 - k, \dots )}} \left(-\calT; \calY_{H(\pi_2)}\right)
.
\intertext{Setting $\calU = \calY_{V(\pi)}$ and $\calV = \calY_{H(\pi)}$, we may view the sum over $\pi_2$ as a sum over all subsequences $\calU$, $\calV \subset \calY$ with the property that $\calU \cup_{p, m - p} \calV \sorteq \calY$:}
\RHS ={} & \sum_{p = 0}^{\min\{l,m\}} \sum_{\substack{\calU, \calV \subset \calY: \\ \calU \cup_{p, m - p} \calV \sorteq \calY}} \hspace{10pt} \sum_{\substack{\pi_1 \in \mathfrak{P}(n - k - l + p, l - p)}} \frac{\Delta\left(\calV; \calS \right) \Delta\left(\calT; \calU \right)}{\Delta \left(\calV; \calU \right) \Delta(\calT; \calS)} \\
& \times (-1)^{|\nu(\pi_1)|} LS_{\mu(\pi_1) + \lambda_{V(\pi_1)} - \left\langle (m - k)^{l(V(\pi_1))} \right\rangle} \left(-\calS; \calU \right) \\
& \times LS_{\nu(\pi_1)' + \lambda_{H(\pi_1)} \cup \lambda_{(n + 1 - k, n + 2 - k, \dots )}} \left(-\calT; \calV \right)
.
\end{align*}
Proposition~\ref{3_prop_visual_interpretation_overlap} allows us to conclude that this expression is equal to the right-hand of the equality in \eqref{3_thm_lapalace_expansion_new_LS_eq}. Hence, the result follows directly from the second overlap identity (\textit{i.e.}\ Theorem~\ref{3_thm_lapalace_expansion_new_LS}).
\end{proof}

\begin{cor} \label{3_cor_second_overlap_id_labeled_walks}
Let $\lambda$ be a partition and let $\calS$ and $\calT$ be sets consisting of $m$ and $n$ variables, respectively. If $\Delta(\calS; \calT) \neq 0$, then
\begin{align} \label{3_cor_second_overlap_id_labeled_walks_eq}
\schur_\lambda(\calS \cup \calT) 
={} & \sum_{\substack{\pi \in \mathfrak{P}(n,m)}} \frac{(-1)^{|\nu(\pi)|} \schur_{\mu(\pi) + \lambda_{V(\pi)}} (\calS) \schur_{\nu(\pi)' + \lambda_{H(\pi)}} (\calT)}{\Delta(\calS; \calT)}
.
\end{align}
\end{cor}

\begin{proof} Owing to the correspondence given in Proposition~\ref{3_prop_visual_interpretation_overlap}, this identity is a direct consequence of Corollary~\ref{3_cor_laplace_expansion_schur_new}.
\end{proof}

\begin{cor} \label{3_cor_overlap_with_subpartitions_LS} Let $m$, $n$, $\tilde n$, $l$ and $q$ be non-negative integers such that $\tilde n \leq q$. Let $\calS$, $\calT$ and $\calY$ be sets of variables of length $m$, $n + \tilde n$ and $q$, respectively, so that $\Delta(\calY) \neq 0$ and $\Delta(\calS, \calT) \neq 0$. Suppose that a partition $\kappa \subset \left\langle (m + n)^l \right\rangle$ satisfies $(m + n, q - \tilde n) \in \kappa$, then
\begin{align*}
LS_{\kappa'}(-(\calS \cup \calT), \calY) ={} & \sum_{p = 0}^{\min\{m, q\}} \sum_{\substack{\calU, \calV \subset \calY: \\ \calU \cup_{p, q - p} \calV \sorteq \calY}} \sum_{\substack{\lambda \subset \left\langle (m - p)^{n + p + l} \right\rangle \\ K \subset [n + p + l] \text{ with } l(K) = l: \\ \sub(\lambda, K) = \kappa}} \frac{\Delta(\calV; \calS) \Delta(\calT; \calU)}{\Delta(\calV; \calU) \Delta(\calT; \calS)} \\
& \times (-1)^{\left| \tilde\lambda_{C_{n + p + l}(K)} \right|} LS_{\lambda' - \left\langle (q - \tilde n)^{m - p} \right\rangle} (-\calS; \calU) \\
& \times LS_{\sub\left(\tilde\lambda, C_{n + p + l}(K)\right)} (-\calT; \calV)
.
\end{align*}
\end{cor}

\begin{proof} Notice that the assumptions on $\kappa$ entail that the $(q, m + n + \tilde n)$-index of $\kappa'$ is $\tilde n$. Thus, the result follows by substituting the correspondence described in Proposition~\ref{3_prop_subpartition_interpretation_overlap} in the right-hand side of the equality in \eqref{3_thm_lapalace_expansion_new_LS_eq}.
\end{proof}

As usual the formula looks considerably nicer specialized to Schur functions, or equivalently to the case $\calY = \emptyset$.

\begin{cor} \label{3_cor_overlap_with_subpartitions_schur} Let $\calS$ and $\calT$ be two sets of variables of lengths $m$ and $n$, respectively, with the property that $\Delta(\calS; \calT) \neq 0$. For any partition $\kappa$ contained in the rectangle $\left\langle (m + n)^l \right\rangle$,
\begin{align*}
\schur_{\kappa'}(\calS \cup \calT) = \sum_{\substack{\lambda \subset \left\langle m^{n + l} \right\rangle \\ K \subset [n + l] \text{ with } l(K) = l: \\ \sub(\lambda, K) = \kappa}} \frac{(-1)^{\left|\tilde\lambda_{C_{n + l}(K)} \right|} \schur_{\lambda'}(\calS) \schur_{\sub\left(\tilde\lambda, C_{n + l}(K)\right)} (\calT)}{\Delta(\calS; \calT)}
.
\end{align*}
\end{cor}

\subsection{A first application}
In this section we present a small application of the second overlap identity. We have called it a first application because our recipe for mixed ratios, which we will present in a forthcoming paper \cite{mixed_ratios}, can be viewed as an application of the first overlap identity for Littlewood-Schur functions.  

Our first application is to derive the dual Cauchy identity from the second overlap identity, or rather from Corollary~\ref{3_cor_second_overlap_id_labeled_walks}. Although the result is classical, this elegant proof seems to be new. The proof relies on the following relationship between the Schur functions indexed by a partition $\lambda$ and it complement $\tilde\lambda$, respectively.

\begin{lem} \label{3_lem_Schur_indexed_by_complement} Let $\calX$ contain $n$ non-zero variables. If a partition $\lambda$ is a subset of the rectangle $\langle m^n \rangle$, then
\begin{align} \label{3_lem_Schur_indexed_by_complement_eq}
\schur_{\tilde\lambda}(\calX) = \schur_\lambda \left( \calX^{-1} \right) \elementary(\calX)^m
.
\end{align}
\end{lem}

\begin{proof} This equality is a fairly immediate consequence of the determinantal definition for Schur functions:
\begin{align*}
\elementary(\calX)^m \schur_\lambda \left( \calX^{-1} \right) ={} & \frac{\elementary(\calX)^m \det \left( x^{-(\lambda_j + n - j)} \right)_{1 \leq j \leq n}}{\det \left( x^{-(n - j)} \right)_{1 \leq j \leq n}} \times \frac{\elementary(\calX)^{n - 1}}{\elementary(\calX)^{n - 1}}
\\
={} & \frac{\det \left( x^{m - \lambda_j + j - 1} \right)_{1 \leq j \leq n}}{\det \left( x^{j - 1} \right)_{1 \leq j \leq n}}
.
\intertext{Inverting the order of the columns in both the numerator and the denominator yields}
\elementary(\calX)^m \schur_\lambda \left( \calX^{-1} \right) ={} & \frac{\det \left( x^{\tilde\lambda_j + n - j} \right)_{1 \leq j \leq n}}{\det \left( x^{n - j} \right)_{1 \leq j \leq n}} = \schur_{\tilde\lambda}(\calX)
. \qedhere
\end{align*}
\end{proof}

\begin{cor} [dual Cauchy identity] Let $\calX$ and $\calY$ be two sets of variables. It holds that
\begin{align} \label{3_cor_dual_cauchy_id_eq}
\sum_\lambda \schur_\lambda(\calX) \schur_{\lambda'}(\calY) = \prod_{\substack{x \in \calX \\ y \in \calY}} (1 + xy).
\end{align}
\end{cor}

\begin{proof} Suppose that $\calX$ and $\calY$ have length $n$ and $m$, respectively. Observe that only partitions $\lambda$ contained in the rectangle $\langle m^n \rangle$ contribute to the sum. As we may assume without loss of generality that no variable in $\calX$ vanishes, we can reformulate the left-hand side in \eqref{3_cor_dual_cauchy_id_eq} as
\begin{align*}
\sum_\lambda \schur_\lambda(\calX) \schur_{\lambda'}(\calY) ={} & \sum_{\lambda \subset \langle m^n \rangle} \elementary \left(\calX^{-1} \right)^{-m} \schur_{\tilde\lambda} \left(\calX^{-1}\right) \schur_{\lambda'}(\calY)
.
\intertext{Using the visual interpretation for $(m,n)$-complementary partitions, this reads}
\sum_\lambda \schur_\lambda(\calX) \schur_{\lambda'}(\calY) ={} & \elementary (\calX)^m \sum_{\pi \in \mathfrak{P}(m,n)} (-1)^{|\nu(\pi)|} \schur_{\mu(\pi)} \left(\calX^{-1}\right) \schur_{\nu(\pi)'}(-\calY)
\end{align*}
where we have also exploited the homogeneity of Schur functions to obtain a signed sum. This sum is essentially equal to the right-hand side in \eqref{3_cor_second_overlap_id_labeled_walks_eq}, specialized to the case that $\lambda$ is the empty partition. Hence, Corollary~\ref{3_cor_second_overlap_id_labeled_walks} allows us to conclude that the above expression is equal to the right-hand side in \eqref{3_cor_dual_cauchy_id_eq}.
\end{proof}

Therefore, the dual Cauchy identity can be viewed as a special case of the second overlap identity: it corresponds to pairs of partitions whose overlap is empty.


\bibliographystyle{alpha}
\bibliography{bib_thesis}

\end{document}